\documentclass{amsart}

\usepackage{enumerate,amsmath,amsthm,tikz-cd,amssymb}
\usepackage{csquotes}

\usepackage{mathtools}
\usepackage{diagbox}
\usepackage{verbatim}

\usepackage{libertine}

\definecolor{cite}{RGB}{44,123,182}
\definecolor{ref}{RGB}{215,25,28}

\usepackage[pdftex,
              pdfauthor={Bronson Lim, Franco Rota},
              pdftitle={Orbifold semiorthogonal decompositions for abelian varieties},
              pdfsubject={derived categories, fourier-mukai functors, abelian varieties, semiorthogonal decompositions},
              pdfkeywords={derived categories, fourier-mukai functors, abelian varieties, semiorthogonal decompositions},
              colorlinks=true,
              citecolor=cite,
              linkcolor=ref 
              ]{hyperref}

\theoremstyle{plain}
\newtheorem{proposition}{Proposition}[section]
\newtheorem{corollary}[proposition]{Corollary}
\newtheorem{conjecture}[proposition]{Conjecture}
\newtheorem{lemma}[proposition]{Lemma}
\newtheorem{theorem}[proposition]{Theorem}

\theoremstyle{definition}
\newtheorem{definition}[proposition]{Definition}
\newtheorem{example}[proposition]{Example}

\theoremstyle{remark}
\newtheorem{remark}[proposition]{Remark}

\newcommand{\DD}{\mathcal{D}}
\newcommand{\TT}{\mathcal{T}}
\newcommand{\FF}{\mathcal{F}}
\newcommand{\EE}{\mathcal{E}}

\newcommand{\BB}{\mathcal{B}}
\newcommand{\KK}{\mathcal{K}}
\newcommand{\Id}{\mathrm{Id}}
\newcommand{\RHom}{\mathbf{R}\mathrm{Hom}}

\DeclarePairedDelimiter{\pair}{\langle}{\rangle}

\newcommand{\fr}{\mathrm{fr}}

\usepackage{adjustbox}
\newcommand{\BigWedge}{\mathord{\adjustbox{valign=m,totalheight=.8\baselineskip}{$\bigwedge$}}}

\usepackage{xcolor}						
\newcommand{\AAA}[1]{\textcolor{red}{#1}}

\begin{document}

\title{Orbifold Semiorthogonal Decompositions for Abelian Varieties}

\author[B. Lim]{Bronson Lim}
\address{BL: Department of Mathematics \\ California State University San
  Bernardino \\ San Bernardino, CA 92407, USA} 
  \email{bronson.lim@csusb.edu}
\author[F. Rota]{Franco Rota}
\address{FR: School of Mathematics \& Statitics \\ University of Glasgow \\ Glasgow, G128QQ, UK} 
\email{franco.rota@glasgow.ac.uk}


\begin{abstract}
  Suppose \(G\) is a finite group acting on an Abelian variety \(A\) such that
  the coarse moduli space \(A/G\) is smooth. Using the recent classification result due to
  Auffarth, Lucchini Arteche, and Quezada, we construct an orbifold semiorthogonal decomposition for
  \(\mathcal{D}[A/G]\) provided \(G = T\rtimes H\) with \(T\) a subgroup of translations and \(H\) is a subgroup of group automorphisms. 
\end{abstract}

\maketitle

\section{Introduction}
\label{sec:intro}

\subsection{Orbifold semiorthogonal decompositions}

Suppose \(G\) is a finite group acting effectively on a smooth quasi-projective
variety \(X\). If
each of the quotients \(\overline{X^\lambda} = X^\lambda/C(\lambda)\) is smooth,
where \(X^\lambda\) is the fixed locus of \(\lambda\in G\) and \(C(\lambda)\) is
the centralizer of \(\lambda\in G\), then there is an orthogonal decomposition
\begin{equation}
  \mathrm{HH}_\ast\left([X/G]\right) = \bigoplus_{\lambda\in G/\sim} \mathrm{HH}_\ast(\overline{X^\lambda})
  \label{eq:HHDecomposition}
\end{equation}
where \(G/\sim\) is the set of conjugacy classes. This decomposition holds for any
additive invariant, see \cite[Remark 1.26]{tvdb-orbifold}. In
\cite{pvdb-equivariant}, the authors conjecture that this
remains true on the derived (or dg) level. Since the derived category decomposes
orthogonally if and only if the variety is not connected, we should not expect a
fully orthogonal decomposition but only a \textit{semiorthogonal decomposition},
see Definition \ref{def:sod}. 

\begin{conjecture}[Orbifold Semiorthogonal Decomposition {\cite[Conjecture A]{pvdb-equivariant}}]
  Suppose \(G\) is a a finite group acting effectively on a smooth variety \(X\)
  and that for all \(\lambda\in G/\sim\) the quotient
  \(\overline{X^\lambda} = X^\lambda/C(\lambda)\) is smooth. Then \(\mathcal{D}[X/G]\) admits a semiorthogonal decomposition where the components \(C_{[\lambda]}\) are in bijection
  with conjugacy classes and \(C_{[\lambda]}\cong
  \mathcal{D}(\overline{X^\lambda})\).
  \label{conj:msodc}
\end{conjecture}
A semiorthogonal decomposition satisfying the conditions of Conjecture \ref{conj:msodc} is called an \textit{orbifold semiorthogonal decomposition}, see Definition
\ref{def:msod}. It is induced by the orbifold structure of $[X/G]$, and it induces the decomposition \eqref{eq:HHDecomposition}, which is motivic in the sense of \cite{tvdb-orbifold} and \cite{toen-motivedm}.
There are several known cases of the conjecture which we list now.
\begin{itemize}
	\item The natural actions of complex reflection groups of types \(A,\ B,\ G_2, F_4\), and \(G(m,1,n)\) acting on \(\mathbb{A}^n\), \cite[Theorem C]{pvdb-equivariant}.
	\item Semidirect product actions of type \(G\rtimes S_n\) on \(C^n\) where \(C\) is a curve and \(G\) is an effective group acting on \(C\), \cite[\S 4.3]{pvdb-equivariant}.
	\item Dihedral group actions on \(\mathbb{A}^2\), \cite[Corollary 6.5.5]{potter-thesis}.
	\item Curves, \cite[Theorem 1.2]{polishchuk-toric}.
	\item Cyclic group quotients, \cite[Theorem 4.1]{kuznetsov-perry-cyclic}, \cite{KPS18}.
\end{itemize}

\subsection{Main result}

Our main result is an affirmative answer to Conjecture \ref{conj:msodc} when
\(A\) is an Abelian variety and \(G = T\rtimes H\) with \(T\) a translation subgroup and \(H\) a subgroup of group homomorphisms.

\begin{theorem}[ = Theorem \ref{thm:msod-abelian-vars-full}]
  Let \(A\) be an Abelian variety and \(G=T\rtimes H\) a finite group of automorphisms of
  \(A\) such that the quotient \(A/G\) is smooth. Then there exists an orbifold semiorthogonal decomposition for
  \(\mathcal{D}[A/G]\). 
  \label{thm:main-msod}
\end{theorem}

\begin{remark}
We did not have to add the hypothesis that \(A^\lambda/C_G(\lambda)\) is
smooth for all \(\lambda\). This is because it is automatically satisfied
provided \(A/G\) is smooth by the classification results in \cite{ga-smooth1, ga-smooth2} which we review in Section \ref{sec:prelims}.
\end{remark}

As an application, we get another proof of the orthogonal decomposition of the
noncommutative motive for \([A/G]\), see \cite[Theorem 1.24, Theorem 1.27]{tvdb-orbifold}. 

\begin{corollary}
  Let \(A\) be an Abelian variety and \(G=T\rtimes H\) a finite group of automorphisms of
  \(A\) such that the coarse moduli space \(A/G\) is smooth. Let \(U\) be
the universal additive invariant. Then the motive
  \(U([A/G])\) decomposes orthogonally:
  \[
    U([A/G])\cong \bigoplus_{\lambda\in G/\sim}U(\overline{A^\lambda}),
  \]
  where \(\overline{A^\lambda} = A^\lambda/C(\lambda)\).
\end{corollary}

\subsection{Outline of Proof}

We rely on the recent work in \cite{ga-smooth1, ga-smooth2}
characterizing smooth quotients of Abelian varieties \(A\) by finite groups. In
particular, there are only finitely many cases where \(A/G\)
is smooth and $T_0(A)$ is an irreducible representation of $G$ (see Section \ref{ssec:quotientsOfAV}). With this classification, we prove Theorem
\ref{thm:main-msod} by using the known orbifold semiorthogonal decomposition for products of curves, by applying the results of \cite{L-P-divisor}, and by constructing a new decomposition for the stacky surface \([E\times E/G(4,2,2)]\), where
\(E=\mathbb{C}/\mathbb{Z}[i]\) is the elliptic curve with automorphism group
\(\mu_4\).

\subsection{Outline of Paper}

Section \ref{sec:prelims} is a preliminary section. Section \ref{sec:msodA2} is devoted to the construction of
an orbifold semiorthogonal decomposition for
\(\mathcal{D}[\mathbb{A}^2/G(4,2,2)]\). This is a local model of the exceptional
case in the classification of smooth quotients of Abelian varieties. The proof
of the global exceptional case is in Section \ref{sec:type-c}. In Section
\ref{sec:msod-abelian}, we complete the proof
of Theorem \ref{thm:main-msod}. Lastly, in Section \ref{sec:examples} we explicitly compute
orbifold semiorthogonal decompositions for Abelian surfaces in the irreducible cases.

\subsection{Conventions and notation}

We work over \(\mathbb{C}\). Unless otherwise stated, all functors are assumed
to be derived. For \(X\) a scheme or stack, by \(\DD(X)\) we mean the bounded
derived category of coherent sheaves on \(X\).

\subsection{Acknowledgements}

The first author would like to thank Robert Auffarth, Aaron Bertram, Giancarlo
Lucchini Arteche, and Alexander Polishchuk for helpful conversations. This
research was part of the RTG grant \#1246989 at the University of Utah.

\section{Preliminaries}
\label{sec:prelims}

We recall preliminaries on equivariant derived categories, orbifold
semiorthogonal decompositions, and smooth quotients of Abelian varieties.
Throughout \(X\) is a smooth quasi-projective variety and \(G\) is a finite
group acting on \(X\). In particular, \(\overline{X} = X/G\) is a quasi-projective
variety. We refer to \cite[Section 4]{bkr-01} for a review of \(G\)-equivariant derived categories and \cite{kuznetsov-survey} for a survey of semiorthogonal decompositions in algebraic geometry.

\subsection{Equivariant derived categories}

For computational purposes it is often convenient to work with the
\(G\)-equivariant derived category of \(X\) instead of the derived category of
\([X/G]\). These are equivalent: \(\mathcal{D}[X/G]\cong \mathcal{D}_G(X)\).

The structure map \(\pi\colon X\to \mathrm{pt}\) is trivially \(G\)-equivariant
and therefore descends to the stack quotients: \(\pi\colon [X/G]\to BG\). Here
\(BG = [\mathrm{pt}/G]\) is the classifying space of \(G\)-bundles. Any finite-dimensional representation \(V\) of \(G\) determines a coherent sheaf on \(BG\)
and hence an object in \(\mathcal{D}(BG)\). For any object
\(\FF\in\mathcal{D}_G(X)\) we therefore have an operation of tensoring by finite-dimensional representations:
\[
  \FF\otimes V:= \FF\otimes\pi^\ast(V).
\]

For any two \(G\)-equivariant sheaves on \(X\), say
\(\EE,\FF\in\mathcal{D}_G(X)\), the set of morphisms is a complex of \(G\)-representations
\[
  \RHom(\EE,\FF)\in\mathcal{D}(BG)
\]
and we have
\[
  \mathrm{Ext}^\ast_{[X/G]}(\EE,\FF)\cong \left( H^\ast\RHom(\EE,\FF)
  \right)^G.
\]

\subsection{Orbifold semiorthogonal decompositions}

\begin{definition}
  Let \(\TT\) be a triangulated category. A semiorthogonal decomposition of
  \(\TT\) is a pair \(\mathcal{A,B}\) of full
  triangulated subcategories of \(\TT\) such that
  \begin{itemize}
    \item \(\mathrm{Hom}_{\TT}(b,a) = 0\) for all \(a\in\mathcal{A},b\in\BB\);
    \item for all \(t\in T\), there exists \(a_t\in \mathcal{A}\) and \(b_t\in \BB\)
      such that
      \[
        b_t\to t\to a_t\to b_t[1]
      \]
      is an exact triangle.
  \end{itemize}
  In this case, we write \(\TT = \langle \mathcal{A},\BB\rangle\).
  \label{def:sod}
\end{definition}

We can iterate Definition \ref{def:sod} to get semiorthogonal decompositions
with an arbitrary finite number of subcategories:
\[
  \TT = \langle \mathcal{A}_1,\ldots,\mathcal{A}_r\rangle.
\]

As suggested in the introduction, the following definition is motivated by the motivic decomposition \eqref{eq:HHDecomposition}.

\begin{definition}
  Suppose a finite group \(G\) acts effectively on a smooth variety \(X\) with smooth quotients $X^{\lambda_i}/C(\lambda_i)$. An
  {\bf orbifold semiorthogonal decomposition} of \(\DD[X/G]\) is the data of
  \begin{itemize}
    \item a total order \(\lambda_1,\ldots,\lambda_r\) on the set of conjugacy classes
      \(G/\sim\);
    \item admissible embedding functors 
      \[
        \Phi_{\lambda_i}\colon \DD(X^{\lambda_i}/C(\lambda_i))\hookrightarrow
        \DD[X/G]
      \]
      for each \(\lambda_i\in G/\sim\);
    \item a semiorthogonal decomposition
      \[
        \DD[X/G] = \langle \mathcal C_1,\ldots, \mathcal C_r\rangle;
      \]
      where \(\mathcal C_i\) is the image of \(\Phi_{\lambda_i}\).
  \end{itemize}
  \label{def:msod}
\end{definition}

\begin{remark}
    An orbifold semiorthogonal decomposition respects the orbifold structure of $[X/G]$: the quotients $X^{\lambda_i}/C(\lambda_i)$ are precisely the connected components of the inertia variety of $[X/G]$.
\end{remark}

\begin{remark}
In all cases where Conjecture \ref{conj:msodc} has been established, the functors $\Phi_{\lambda_i}$ are \textit{linear} over \(\DD(X/G)\). That is, tensoring with pull-backs from $X/G$ preserves the image of $\Phi_{\lambda_i}$. 
While linearity does not appear as a condition in the Conjecture, it is used to simplify several base change arguments in \cite{L-P-divisor}. We point out, moreover, that the functors defining the decompositions of Theorem \ref{thm:main-msod} are linear (this follows from the definition of the Fourier-Mukai kernels in \S \ref{ssec:FM_functors_A2}).
\end{remark}

\begin{example}
  The simplest example of an orbifold semiorthogonal decomposition is when \(\mu_2\) acts on a smooth variety \(X\) fixing a
  smooth divisor \(Y\) pointwise. Then \(\overline{X} = X/\mu_2\) is smooth. Set
  \(\iota\colon Y\hookrightarrow X\) to be the inclusion, which is \(\mu_2\)-equivariant, and \(\pi \colon X\to
  \overline{X}\) the quotient mapping. Then it is straightforward to check that there is a
  semiorthogonal decomposition
  \[
    \DD[X/\mu_2] = \left\langle \pi^\ast\DD\left(\overline{X}\right),\iota_\ast\DD(Y)\right\rangle.
  \]
  This semiorthogonal decomposition is \(\DD\left(\overline{X}\right)\)-linear and hence defines an orbifold semiorthogonal
  decomposition.

  Since \(\mu_2\) acts trivially on \(\iota_\ast \DD(Y)\), we can {\it tensor by
  irreducible representations} to get different objects. Let \(\chi\colon
  \mu_2\to \mathbb{C}^\ast\) be the unique nontrivial representation. Then for
  any object \(\iota_\ast\FF\in \iota_\ast\DD(Y)\), we can consider the object \(\iota_\ast\FF\otimes\chi\).
  This is still an object in \(\DD[X/\mu_2]\) but no longer in \(\iota_\ast\DD(Y)\). More precisely, the \(G\)-equivariant adjunction formula shows the Serre functor comes with an additional tensoring by
  \(\chi\). Applying Serre duality (which in this context coincides with a \textit{mutation}, see \cite{bondal-kapranov-89}) we have another orbifold semiorthogonal decomposition
  \[
    \DD[X/\mu_2] = \langle \iota_\ast\DD(Y)\otimes\chi,\pi^\ast\DD(\overline{X})\rangle.
  \]

  This example indicates one of the central problems in solving this conjecture. A
  canonical ordering on embeddings doesn't exist: one could have fully-faithful
  embedding functors but they (or any reordering) are not necessarily
  semiorthogonal.
\end{example}

\subsection{Smooth quotients of Abelian varieties}
\label{ssec:quotientsOfAV}

We recall the main results of \cite{ga-smooth1,ga-smooth2}. Let \(G\) be a finite group of automorphisms of an Abelian variety \(A\) such that \(G = T\rtimes H\) with \(T\) a group of translations and \(H\) a group of group homomorphisms.

\begin{theorem}[{\cite[Theorem 2.7]{ga-smooth1}}]
\label{thm:reducible_actions}
  Suppose \(G\) acts without translations and \(\dim(A^G) = 0\)\footnote{If $A^G$ has positive dimension, the theorem does not necessarily hold, but $A/G$ fibers over an abelian variety with smooth fibers satisfying Theorem \ref{thm:reducible_actions} \cite[Prop. 2.9]{ga-smooth1}. In the notation of Section \ref{ssec:pos-dim}, the base of the fibration is $A_0/\Delta$, and the fibers are isomorphic to $P_G/G$.}. If \(G\) acts
  reducibly on \(T_e(A)\), then there exists direct product decompositions
  \(A\cong A_1\times\cdots \times A_k\) and \(G\cong G_1\times\cdots \times
  G_k\) where the action is diagonal and \(G_i\) acts irreducibly on
  \(T_e(A_i)\).
  \label{thm:direct-prod-decomp}
\end{theorem}

By \cite[Lemma 2.4.1]{L-P-divisor}, it suffices to construct an orbifold semiorthogonal decomposition for the factors $[A_i/G_i]$, so we can assume that $G$ acts irreducibly.


\begin{theorem}[{\cite[Theorem 1.1]{ga-smooth1},\cite[Theorem 1.1]{ga-smooth2}}]
\label{thm:ALAQClassification}
  Suppose \(G\) fixes the identity, \(\dim(A^G) = 0\), \(G\) acts
  irreducibly on \(T_e(A)\), and \(A/G\) is smooth. If \(\dim(A)\geq 3\), then
  there exists an elliptic curve \(E\) such that \(A\cong E^n\) and either
  \begin{enumerate}[(A)]
    \item \(G\cong C^n\rtimes S_n\) where \(C\neq 1\) is a finite cyclic group acting
      on \(E\) so that \(C^n\) acts diagonally and \(S_n\) permutes the factors;
      or
    \item \(G\cong S_{n+1}\) and acts on \(A = \{ (x_1,\ldots,x_{n+1})\in
      E^{n+1}\mid x_1+\cdots+x_{n+1} = e\}\) by permutations.
  \end{enumerate}
  If \(\dim(A) = 2\), then both (A) and (B) can happen and there is a third
  case: 
  \begin{enumerate}[(A)]
    \addtocounter{enumi}{2}
    \item \(E\cong \mathbb{C/Z}[i]\) and \(G=G(4,2,2)\cong K\rtimes S_2\), where \(K
      = \{ (i^a,i^b)\mid a+b\equiv_20\}\). 
  \end{enumerate}
  \label{thm:zero-fixed}
\end{theorem}

The case $\dim(A)=0$ is well-known: every pair $(A,G)$ has $\mathbb{P}^1$ as a quotient. We comment on this class of examples in \S \ref{ssec:examples_curves}.

We will refer to the cases of Theorem \ref{thm:ALAQClassification} as irreducible quotients of Types (A), (B), and (C).
Orbifold semiorthogonal decompositions of Type (A) have been constructed in \cite{pvdb-equivariant}.
Types (B) and (C) have not yet been considered. We will construct one for Type (C) in Section
\ref{sec:type-c}. We will show they exist for type (B) in Theorem \ref{thm:fixing-origin}.

\section{orbifold semiorthogonal decomposition for \texorpdfstring{\([\mathbb{A}^2/G(4,2,2)]\)}{A2/G(4,4,2)}}
\label{sec:msodA2}

Before constructing an orbifold semiorthogonal decomposition for Type (C), we start by studying the local case of $G(4,2,2)$ acting naturally on $\mathbb{A}^2$.

\subsection{Representation Theory of \texorpdfstring{\(G(4,2,2)\)}{G(4,2,2)}}
\label{ssec:RepThyG}

We will need the representation theory of \(G(4,2,2)\). We recall this information now. The complex reflection group
\(G(4,2,2)\) is constructed as follows. Fix \(\xi = \sqrt{-1}\) to be a primitive
fourth root of unity, let \(\mu_4^2 = \langle (\xi,1),(1,\xi)\rangle\) and \(K\)
the subgroup given by
\[
  K = \{ (\xi^a,\xi^b)\in\mu_4^2\mid a+b\equiv_20\}.
\]
The symmetric group \(S_2 = \langle \sigma\rangle\) acts on \(K\) by
\[
  \sigma(\xi^a,\xi^b) = (\xi^b,\xi^a)
\]
and we set
\[
 G \coloneqq G(4,2,2) = K\rtimes S_2.
\]
The group \(\mu_4^2\) acts on \(X\coloneqq \mathbb{A}^2\) via matrix multiplication
and $S_2$ acts by permuting coordinates. This defines an action of \(G\) on
\(X\), which is an irreducible two-dimensional representation. We call it the natural action of $G$ and denote it \(V\). 

More explicitly, \( (-1,1,1), (-\xi,\xi,1), (1,1,\sigma)\) are generators for $G$, and the natural action is defined by matrix multiplication with
\begin{equation}
    \label{eq:NaturalAction}
    (-1,1,1)\mapsto \begin{pmatrix}
        -1&0\\0&1
    \end{pmatrix}, \quad (-\xi,\xi,1)\mapsto \begin{pmatrix}
        \xi&0\\0&\xi
    \end{pmatrix},    \quad     (1,1,\sigma)\mapsto \begin{pmatrix}
        0&1\\1&0
    \end{pmatrix}. 
\end{equation}

Note that \(V\) is not self-dual, as \(\chi_V(\xi,\xi,1)=2\xi\) is purely
complex. We have \(X=\mathbb{A}^2\cong \mathrm{Spec}(\mathrm{Sym}(V^\vee))\). Explicitly, we will write $X=\mathrm{Spec}(\mathbb{C}[x,y])$ and the action of $G$ on linear terms is the dual representation
\begin{equation}
    \label{eq:dualRep}
    (\xi^{a},\xi^b,1)\, x = \xi^{-a} x,  \quad 
      (\xi^{a},\xi^b,1)\, y = \xi^{-b} y, \quad       (1,1,\sigma) (x) = y.
\end{equation}

\begin{remark}
The group $G$ is also known in the literature as the \textit{Pauli group}. It is a central product of $C_4$ and $D_4$, and it is the $13^{\mathrm{th}}$ group of order 16 in the small group database \cite{Dok}.
We gave the definition via complex reflection groups, which seems natural from the point of view of equivariant geometry. We point out, however, that the natural action is not the only matrix action of $G$ in this work (see Section \ref{sec:type-c}).
\end{remark}

\begin{proposition}
  There are ten irreducible representations of \(G\). Eight of the
  irreducible representations are one-dimensional and the remaining two are two-dimensional. Moreover, the one-dimensional irreducible representations are 2-torsion.
  \label{prop:irrepsg422}
\end{proposition}

\begin{proof}
  The element \((-1,-1,1)\) lies in the center of \(G\) and the quotient by it is an elementary Abelian 2-group of order 8:
  \[
    G/\langle (-1,-1,1)\rangle = \mu_2^3.
  \]
  Indeed, define a homomorphism \(\rho\colon G\to \mu_2^3\) by
  \[
    \rho(\xi^a,\xi^b,\sigma^c) = (\xi^{a+b},\xi^{a-b},\mathrm{sgn}(\sigma^c)).
  \]
  It is a homomorphism as \(\xi^{a-b}=\xi^{b-a}\) in this case. The kernel is \(
  (-1,-1,1)\) and it is evidently surjective. It follows that there are eight
  irreducible one-dimensional representations. Since the sum of the squares of
  the dimensions of the irreducible representations must be \(|G| = 16\)
  we have
  \[
    8 + \dim(V_1)^2+\ldots+\dim(V_k)^2 = 16
  \]
  and thus \(k=2\) and \(\dim(V_1) = \dim(V_2) = 2\). These two-dimensional representations are \(V\) and \(V^\vee\).
\end{proof}

Let \(\rho\colon G\to \mu_2^3\) to be the surjective homomorphism from
the proof of Proposition \ref{prop:irrepsg422} and \(\pi_{i+1}\colon
\mu_2^3\to\mathbb{C}^\ast\) the projection onto the \(i\)th factor character.
Then set
\[
  \chi_i = \pi_i\circ \rho\colon G\to\mathbb{C}^\ast.
\]
Explicitly
\begin{align*}
  \chi_2(\xi^a,\xi^b,\sigma^c) &= \xi^{a+b} \\
  \chi_3(\xi^a,\xi^b,\sigma^c) &= \xi^{a-b} \\
  \chi_4(\xi^a,\xi^b,\sigma^c) &= \mathrm{sgn}(\sigma^c). \\
\end{align*}

The characters \(\chi_2,\chi_3,\text{ and }\chi_4\) have order 2 and commute with one another. They generate  the character group of $G$, since they generate a subgroup $(\mu_2)^3$ of the same size.  

Note that we have
\begin{equation}\label{eq:WedgeVCharacter}
  \bigwedge^2V = \chi_2\chi_4,  
\end{equation}
since $\det \lambda = \xi^{a+b} \mathrm{sgn}(\sigma^c)$, for $\lambda$ the matrix associated to $(\xi^a,\xi^b,\sigma^c)$ via the natural action \eqref{eq:NaturalAction}.

\begin{proposition}
  There are isomorphisms of representations
  \begin{align*}
    V\otimes V&\cong \chi_2\chi_4\oplus \chi_2\chi_3\oplus \chi_2\chi_3\chi_4\oplus
    \chi_2\\
    V^\vee&\cong V\otimes\chi_2\chi_4\cong V\otimes\chi_2\chi_3\cong
    V\otimes\chi_2\chi_3\chi_4\cong V\otimes\chi_2 \\
    V\otimes V^\vee&\cong \chi_3\chi_4\oplus \mathbf{1}\oplus \chi_4\oplus
    \chi_3.
  \end{align*}
  \label{prop:tensorsquareV}
\end{proposition}

\begin{proof}
  For the isomorphism on the first line we exhibit an explicit decomposition. Pick the
  standard basis \(e_1,e_2\) for \(V\). Then, for example,
  \[(\xi^a,\xi^b,\sigma^c) (e_1\otimes e_1+e_2\otimes e_2) = \xi^{2a}(e_1\otimes e_1)+ \xi^{2b}(e_2\otimes e_2) = \xi^{2a}(e_1\otimes e_1+e_2\otimes e_2) \]
  because $2a \equiv_4 -2a \equiv_4 2b$ by definition of $K$. Since $\xi^{2a} = \chi_2\chi_3 (\xi^a,\xi^b,\sigma^c)$, we have 
  \[  \chi_2\chi_3 = \mathrm{Span}\{e_1\otimes e_1+e_2\otimes e_2\}  \]
  Arguing similarly, we obtain:
  \begin{align*}
    \chi_2\chi_4 &= \mathrm{Span}\{e_1\otimes e_2-e_2\otimes e_1\} \\
    \chi_2\chi_3\chi_4 &= \mathrm{Span}\{e_1\otimes e_1-e_2\otimes e_2\} \\
    \chi_2 &= \mathrm{Span}\{e_1\otimes e_2+e_2\otimes e_1\}.
  \end{align*}

  The first isomorphism on the second line is a standard identity (using \eqref{eq:WedgeVCharacter}). For the others, recall the dual action described in \eqref{eq:dualRep}.
  Then, suppose for example that $e_1',e_2'$ are a basis of $V\otimes \chi_2\chi_3$ and check 
  \begin{align*}
      (\xi^{a},\xi^b)\, e_1' = \xi^{-a} e_1', & \quad &
      (\xi^{a},\xi^b)\, e_2' = \xi^{-b} e_1', & \quad &
      \sigma (e_1') = e_2'
  \end{align*}
(again working modulo 4), which shows that $(x,y)\mapsto (e_1',e_2')$ induces $V^\vee \cong V\otimes\chi_2\chi_3$. The other isomorphisms are shown similarly.

  The last isomorphism follows from the first two.
\end{proof}

The following is straightforward. We include it for labeling purposes.

\begin{proposition}\label{prop:ClassesG}
  The conjugacy classes of \(G\) are \footnote{Identifying the group $G$ with the Pauli group $\{ \pm \sigma_k, \pm \xi \sigma_k\}_{k=0,1,2,3}$ via the natural action, the reader will readily recognize the 10 conjugacy classes as $\{ \mathrm{Id} \},\{ -\mathrm{Id} \},\{ \xi\mathrm{Id} \},\{ -\xi\mathrm{Id} \}$ and $\{ \pm \sigma_k \}, \{ \pm \xi \sigma_k\}$ for $k=1,2,3$.}:
  \begin{equation*}
  \begin{aligned}
    D_1 &= \{\Id\} & 
    D_2 &= \{ (1,1,\sigma),(-1,-1,\sigma)\} \\
    D_3 &= \{ (1,-1,1),(-1,1,1)\} & 
    D_4 &= \{ (\xi,-\xi,\sigma),(-\xi,\xi,\sigma)\} \\
    D_5 &= \{ (-1,-1,1)\} & 
    D_6 &=\{ (\xi,\xi,1)\} \\
    D_7 &= \{ (\xi,\xi,\sigma),(-\xi,-\xi,\sigma)\} &
    D_8 &= \{ (-\xi,-\xi,1)\} \\
    D_9 &= \{ (1,-1,\sigma),(-1,1,\sigma)\} & 
    D_{10} &= \{(\xi,-\xi,1),(-\xi,\xi,1)\}.
    \end{aligned}
  \end{equation*}
   The centralizers of each conjugacy class are:
  \begin{equation*}
  \begin{aligned}
    C_1 &= G & 
    C_2 &= \langle (\xi,\xi,1),(1,1,\sigma)\rangle\cong \mu_4\times
    \mu_2 \\
    C_3 &= K & 
    C_4 &= \langle (\xi,\xi,1),(\xi^2,1,\sigma)\rangle
    \cong \mu_4\times\mu_2\\
    C_5 &= G & 
    C_6 &= G \\
    C_7 &= \langle (\xi,\xi,1),(1,1,\sigma)\rangle \cong\mu_4\times \mu_2& 
    C_8 &= G \\
    C_9 &= \langle (\xi,\xi,1),(1,\xi^2,\sigma)\rangle \cong\mu_4\times \mu_2&
    C_{10} &= K.
  \end{aligned}
  \end{equation*}
  For all $i=1,...,10$, let $\lambda_i$ denote a representative of $D_i$. Then, the fixed loci $X_i\coloneqq X^{\lambda_i}$ are 
  \begin{equation}
  \begin{aligned}
    X_1 &= X, &
    X_5=...=X_{10}=\{0\},
  \end{aligned}
  \end{equation}
  and $X_2,X_3,X_4$ are lines intersecting at the origin. 
  \label{prop:conjclassesg422}
\end{proposition}

\subsection{Fourier-Mukai kernels}
\label{ssec:FM_functors_A2}

For \(i=1,\ldots,10\), we define Fourier-Mukai kernels by constructing \(H_i\times G\)-invariant subschemes \(Z_i\) of \(  X_i\times X \) as follows. Define
\[
  Z_i = \bigcup_{g\in G}g^\ast\Gamma_i
\]
where \(\Gamma_i\) is the graph of the inclusion \(X_i\hookrightarrow X\),
and we give \(Z_i\) its reduced scheme structure. Then \(Z_i\) is an \(H_i\times G\)-
invariant subscheme and so \(\mathcal{O}_{Z_i}\) can be taken as a kernel for the Fourier-Mukai functor 
\[
    \Phi_i = \Phi_{\mathcal{O}_{Z_i}}\circ\pi_i^\ast\colon\mathcal{D}(\overline{X_i})\to\mathcal{D}[X/G(4,2,2)],
\]
where \(\pi_i\colon X_i\to \overline{X_i}\) is the quotient mapping. Equivalently, let \(\overline{Z}_i\) be the image of \(Z_i\) under the quotient \( (\pi_i,\mathrm{Id})\colon X_i\times X\to \overline{X_i}\times X\), then
\[
	\Phi_i = \Phi_{\mathcal{O}_{\overline{Z}_i}}.
\]
We note that if \(X_i\) is a point, then \(Z_i\) is also a point and the corresponding $\Phi_i$
is simply the pushforward along the inclusion of the origin which is \(G\)-invariant. If \(X_i\) is a line,
then \(Z_i\) is the union of two lines in \(X_i\times X\cong\mathbb{A}^3\) as
the centralizer has index two. If \(X_i = X\), then $\Phi_1=\pi^*$, the pull-back functor from the coarse space.

\smallskip

We proceed to compute the images of points via the functors $\Phi_i$, for $i=1,2,3,4$. We begin by fixing some notation. 

For $p\in \overline{X_i}=X_i/C_i$, we define a subset $L_p\subset X$ as follows: pick a point $q\in \pi_i^{-1}(p)\subset X_i$. Regard $q$ as a point of $X$ under the inclusion $X_i \hookrightarrow X$ and let $L_p$ be the $G$-orbit of $q$. It is straightforward to check that $L_p$ only depends on $p$. 
For example, in the case $i=1$, $L_p$ is precisely the orbit represented by $p \in \overline{X_1}=X/G$.

We will write $0\in \overline{X_i}$ for the image of $0\in X_i$ for all $i=1,...,10$.


\begin{proposition}
    \label{prop:image of points1}
     We have one of the following cases:
    \begin{itemize}
\item $p=0$ and $q=0$. Then 
\[ \Phi_{1}(\mathcal{O}_0) = \mathbb{C}[x,y]/(x^2y^2,x^4+y^4). \]
\item $p$ is a free orbit and $q$ has trivial stabilizer. Then $L_p$ consists of 16 distinct points and 
\[ \Phi_{1}(\mathcal{O}_p)=\mathcal{O}_{L_p}. \]
\item $p$ is not a free orbit and $q$ has a $\mu_2$ stabilizer. In this case the orbit $L_p$ of $p$ contains 8 points, and $\Phi_{1}(\mathcal{O}_p)$ is a sheaf supported on $L_p$ with a nilpotent of order 2 at every point.
\end{itemize}
\end{proposition}
\begin{proof}
The functor $\Phi_{1}$ is pull-back $\pi^*$ from the coarse moduli space. 
The maximal ideal of the origin pulls back to the ideal of $G$-invariant polynomials $(x^2y^2,x^4+y^4)$, and a point $p\neq 0$ pulls back to its orbit. 

If $q$ has trivial stabilizer, then $L_p$ consists of 16 points, all belonging to the locus where the quotient map is a local isomorphism.

If $q$ has a $\mu_2$ stabilizer, then there are 8 points in its orbit, and the pull-back $\pi^*(\mathcal{O}_{\pi(p)})$ is non-reduced of length 2: locally around $q$ the quotient map $X\to \overline{X}$ is given by $\mathbb{C}[u,v^2] \to \mathbb{C}[u,v]$ and $p=\pi(q)$  pulls back to $\mathbb{C}[u,v]/(u,v^2)$.

There are no other possibilities for the stabilizer of a point $0\neq q\in X$: any such $q$ either has trivial stabilizer, or belongs to a translate of $X_i$ for some $i=2,3$, or $4$. In the latter case the stabilizer is $\mu_2$, generated by one of the elements of $D_i$. 
\end{proof}

\begin{proposition}
    \label{prop:image of points234}
    For $i=2,3,4$ and $p=0\in \overline{X_i}$, we have
     \[
     \Phi_{i}(\mathcal{O}_0) = \begin{cases}
       \mathbb{C}[x,y]/(x^2y^2,x^2-y^2) & i = 2\\
       \mathbb{C}[x,y]/(xy,(x+y)^4) & i = 3 \\
       \mathbb{C}[x,y]/(x^2y^2,x^2+y^2) & i = 4.
    \end{cases}
  \]
 For any other point $0\neq p\in \overline{X_i}$, there are four distinct points \(q_1,q_2,q_3,q_4\) in the preimage
  of \(p\). Then
  \[
    L_p = \{q_1,q_2,q_3,q_4,\gamma q_1,\gamma q_2,\gamma q_3,\gamma q_4\}
  \]
  where \(\gamma\) is a nontrivial coset representative of $C_i$ in $G$. In this case, 
  \[
    \Phi_{i}(\mathcal{O}_p) \cong \mathcal{O}_{L_p}.
  \]  
\end{proposition}
\begin{proof}
The second statement ($p\neq 0$) follows since the centralizers have order 8 but do not act faithfully on $X_i$. Hence, the pull-back of \(\mathcal{O}_p\) under the
  quotient map \(X_i\to \overline{X_i}\) is 
  \[
    \pi_i^\ast(\mathcal{O}_p)\cong \bigoplus_{j=1}^4\mathcal{O}_{q_j}.
  \]
  Finally, let \(\gamma\) be a coset representative of the only nontrivial
  coset. Then, when we equivariantize to build $Z_i$, $\gamma$ just translates each of the four points $q_j$ to $\gamma q_j$.

  The first statement is more involved. We will only do the computation for \(i
  =3\) as the other computations are analogous under a linear change of
  coordinates. Pick $\lambda \coloneqq \lambda_3 = (1,-1,1)$. Thus \(X_3\times X\cong
  \mathbb{A}^1_z\times\mathbb{A}^2_{x,y}\) and $\Gamma_\lambda$ has equation $z=x$. 
  We can take \(\rho\coloneqq (1,1,\sigma)\)
  as a nontrivial coset representative for \(C(\lambda)=C_3\) in \(G\). Thus \(Z_\lambda =
  \Gamma_\lambda\cup \rho^\ast\Gamma_\lambda\) with its reduced scheme
  structure. As modules we have
  \[
    H^0(\mathcal{O}_{Z_\lambda})\cong \mathbb{C}[z,x,y]/(z-x-y,xy).
  \]

  Let \(X_3\to \overline{X_3}\) denote the quotient map. Note that
  \(C(\lambda)/(\lambda) = K/(\lambda) \cong \mu_4\), and it acts in the natural way on \(X_3\). In
  particular if \(x\) is the coordinate on \(X_3\), then the quotient map is
  \(x\mapsto x^4\). It follows that for \(0\in\overline{X_3}\cong\mathbb{A}^1\), we
  have
  \[
    \pi_3^\ast(\mathcal{O}_0) = \mathcal{O}_{0,4}
  \]
  where $\mathcal{O}_{0,4}$ is a non-reduced
  zero-dimensional subscheme of length 4 of $X_3$, concentrated at the origin. In other words, the complex \(z^4 \colon \mathcal{O}_{X_3}\rightarrow\mathcal{O}_{X_3}\) is a resolution of $\mathcal{O}_{0,4}$ on $X_3$. Pull back the resolution to $X_3 \times X$. 
  Tensoring with
  \(\mathcal{O}_{Z_\lambda}\) yields
  \[
    \mathcal{O}_{Z_\lambda}\xrightarrow{z^4}\mathcal{O}_{Z_\lambda}.
  \]
  Since the map $z^4$ is injective, this complex is quasi-isomorphic to
  \[
    \mathbb{C}[z,x,y]/(z-x-y,xy,z^4),
  \]
  which pushes forward to \(\mathbb{C}[x,y]/(xy,(x+y)^4)\). 
\end{proof}

\begin{corollary}
Let $R$ denote the regular representation of $G$. Then:
  \begin{align*}
  \Phi_{1}(\mathcal{O}_0) \cong  R.
\end{align*}
Moreover we have the following isomorphisms of representations:
\begin{align*}
    \Phi_{2}(\mathcal{O}_0) \cong & \mathbf{1}\oplus V^\vee\oplus
    V\oplus \chi_2 \oplus \chi_2\chi_3\oplus \chi_3 \\
    \Phi_{3}(\mathcal{O}_0) \cong & \mathbf{1}\oplus V^\vee\oplus
    V\oplus \chi_2\chi_3\oplus \chi_2\chi_3\chi_4\oplus \chi_4 \\
    \Phi_{4}(\mathcal{O}_0) \cong & \mathbf{1}\oplus V^\vee\oplus
    V\oplus \chi_2\oplus \chi_2\chi_3\chi_4 \oplus \chi_3\chi_4
  \end{align*}
  \label{cor:rep-decomp-origins}
\end{corollary}

\begin{proof}
This is straightforward. For example, the 16 monomials spanning  \(\Phi_{1}(\mathcal{O}_0)\) as a \(\mathbb{C}\)-vector space can be grouped in representations as follows:
  \begin{equation*}
  \begin{aligned}
   \mathbb{C}\{1\} & \cong\mathbf{1}, & \mathbb{C}\{x,y,x^5,x^4y\} &\cong (V^\vee)^{\oplus 2},\\
    \mathbb{C}\{x^3,x^2y,xy^2,y^3\} &\cong V^{\oplus 2}, &
     \mathbb{C}\{x^2,y^2\} &\cong \chi_2\chi_3 \oplus \chi_2\chi_3\chi_4,\\
     \mathbb{C}\{x^3y, xy^3\}&\cong \chi_3\oplus \chi_3\chi_4, & 
     \mathbb{C}\{xy\}&\cong \chi_2,\\
      \mathbb{C}\{x^4\}&\cong \chi_4, & \mathbb{C}\{x^5y\}&\cong \chi_2\chi_4.
  \end{aligned}
  \end{equation*}

\end{proof}

\begin{proposition}
  For each \(i=1,\ldots,10\), the Fourier-Mukai functors
  \[
    \Phi_i\colon\mathcal{D}(\overline{X_i})\to\mathcal{D}[X/G(4,2,2)]
  \]
  are fully-faithful.
  \label{prop:ffcomponents}
\end{proposition}

\begin{proof}
We use the Bondal-Orlov fully-faithfulness criterion\footnote{The criterion applies to the current case of a quasi-projective domain and an arbitrary triangulated target by  combining \cite{LM17} and \cite[Theorem 2.7.1]{Lim21}.}. 
 The functor $\Phi_1$ is pull-back from \(X/G(4,2,2)\), which is
  fully-faithful by projection formula. For the other cases, we apply the Bondal-Orlov fully-faithfulness criterion. If \(i=5,\ldots,10\), $\Phi_i$ is just push-forward along the inclusion $\{0\} \hookrightarrow X$, which
  is fully-faithful since the structure sheaf of the origin is exceptional:
 \begin{equation}
 \label{eq:EndomorphismsO_0}
    \mathrm{Ext}_X^\ast(\mathcal{O}_0,\mathcal{O}_0)^G = \left(\mathbf{1}[0]
    \oplus V[-1]\oplus \BigWedge^2V[-2]\right)^G = \mathbf{1}[0].
 \end{equation}

  The computations for \(i=2,3,4\) are all similar. For notational convenience,
  we will compute the case of \(i=3\). In this case, we have two types of
  points. If \(p\neq 0\), then we look at the orbit of a preimage under the
  quotient map. By Proposition \ref{prop:image of points234}, the orbit is a reduced subscheme of length 8 \(L_p=\bigsqcup_{j=1}^8 p_j\). Moreover,
  we have
  \[
    \mathrm{Ext}_X^\ast(\mathcal{O}_{L_p},\mathcal{O}_{L_p})^G =
    \left(\bigoplus_{j=1}^8\mathrm{Ext}_X^\ast(\mathcal{O}_{p_j},\mathcal{O}_{p_j})\right)^G.
  \]
  Clearly when \(\ast = 0\) the \(G\)-action permutes the hom-sets and so the
  8-dimensional hom-space has only one copy of the trivial representation. If
  \(\ast = 2\), then we have
  \[
    \left(\bigoplus_{j=1}^8T_{p_j}X_i\otimes N_{p_j}X_i\right)^G
  \]
  and since \(X_3\) is codimension 1 it must be that \(\lambda_3\) acts
  nontrivially on the normal bundle. Thus there are no invariants and we have the necessary vanishing.

  For \(p = 0\), recall from Proposition \ref{prop:image of points234} that \(\Gamma(\Phi_{3}(\mathcal{O}_0))\cong
  \mathbb{C}[x,y]/(xy,(x+y)^4)\). Hence,
  \[
    \mathrm{Hom}(\Phi_{3}(\mathcal{O}_0),\Phi_{3}(\mathcal{O}_0))^G\cong
    \mathbb{C}
  \]
  as any map is determined by the image of \( 1\) and \(
  \Phi_{3}(\mathcal{O}_0)^G\cong \mathbf{1}\). For the vanishing of the
  second Ext group we use Serre duality\footnote{As illustrated in \cite[\S 4.3]{bkr-01}, the canonical bundle with its $G$-structure gives a Serre functor in the derived category of compactly supported, $G$-equivariant sheaves on a quasi-projective variety. On an affine variety, these are $G$-equivariant sheaves with finite support. We will always refer to this version of Serre duality in the rest of the paper.}:
  \[
    \mathrm{Ext}^2(\Phi_{3}(\mathcal{O}_0),\Phi_{3}(\mathcal{O}_0))\cong
    \mathrm{Hom}(\Phi_{3}(\mathcal{O}_0),\Phi_{3}(\mathcal{O}_0)\otimes\BigWedge^2V)^\vee
  \]
  and by Corollary \ref{cor:rep-decomp-origins}
  \[
    \Phi_{3}(\mathcal{O}_0)\cong \mathbf{1}\oplus V\oplus V^\vee
    \oplus \chi_2\chi_3\oplus \chi_2\chi_3\chi_4\oplus \chi_4
  \]
  and so \( (\Phi_{3}(\mathcal{O}_0)\otimes\BigWedge^2V)^G = 0\). This
  implies the vanishing of the second Ext group and completes the proof.
\end{proof}

\subsection{Assembling the Orbifold Semiorthogonal Decomposition}
\label{ssec:assembling_MSOD_A2}

In this section we construct an orbifold semiorthogonal decomposition of $[X/G]$ using the functors $\Phi_i$. We start by observing that the positive-dimensional loci have semiorthogonal images:

\begin{lemma}
\begin{enumerate}
\item Let $i=2,3,4$ and $p\in X_i$. We have
  \[ \mathrm{Ext}^\ast(\Phi_{i}(\mathcal{O}_p),\Phi_{1}(\mathcal{O}_p))^G  = 0 \]
  \item For $i,j=2,3,4$, $j\neq i$, 
  \[\mathrm{Ext}^\ast(\Phi_{i}(\mathcal{O}_0),\Phi_{j}(\mathcal{O}_0))^G  = 0. \]
\end{enumerate}
\label{lem:orthogDim1and2}
\end{lemma}

\begin{proof}
Recall that $\Phi_{1}=\pi^*$, and observe
\begin{align*}
\mathrm{Ext}^{2-\bullet}(\Phi_{i}(\mathcal{O}_p),\pi^*(\mathcal{O}_p)) &\cong \mathrm{Ext}^\bullet(\pi^*(\mathcal{O}_p),\Phi_{i}(\mathcal{O}_p)\otimes \BigWedge^2V)^\vee \\
&\cong \mathrm{Ext}^\bullet(\mathcal{O}_p,\pi_*(\Phi_{i}(\mathcal{O}_p)\otimes \BigWedge^2V))^\vee \\
&=0,
\end{align*}
by Serre duality, and where the last equality follows because  $\Phi_{i}(\mathcal{O}_p)\otimes \BigWedge^2V$ has no invariants (see Corollary \ref{cor:rep-decomp-origins} for $p=0$, and Proposition \ref{prop:image of points234} for $p\neq 0$).

The argument is slightly more involved for the second statement. We only illustrate the case $i=2$ and $j=3$, the others are analogous. A resolution of  $\Phi_{2}(\mathcal{O}_0)$ is the Koszul complex 
\[ \mathbb C[x,y]\otimes \chi_2\chi_3\chi_4 
\xrightarrow{\begin{pmatrix}
y^2-x^2 \\ x^2y^2
\end{pmatrix}}
\begin{array}{c}\mathbb C[x,y]\\
\oplus\\
\mathbb C[x,y]\otimes\chi_2\chi_3\chi_4\end{array}
\xrightarrow{\cdot\begin{pmatrix}
x^2y^2 & x^2-y^2
\end{pmatrix}}
\mathbb C[x,y] \]
where $\chi_2\chi_3\chi_4$ is the weight of $x^2-y^2$.
Applying $\mathrm{Hom}(-,\Phi_{3}(\mathcal{O}_0))$ yields 
\begin{equation}
\label{eq:koszul1Dim}
\Phi_{3}(\mathcal{O}_0)
\xrightarrow{\begin{pmatrix}
0 \\ y^2-x^2
\end{pmatrix}}
\begin{array}{c}\Phi_{3}(\mathcal{O}_0)\\
\oplus\\
\Phi_{3}(\mathcal{O}_0)\otimes\chi_2\chi_3\chi_4\end{array}
\xrightarrow{\cdot\begin{pmatrix}
x^2-y^2 & 0
\end{pmatrix}}
\Phi_{3}(\mathcal{O}_0)\otimes \chi_2\chi_3\chi_4.
\end{equation}
Recall from Proposition \ref{prop:image of points234} that $\Phi_{3}(\mathcal{O}_0) \cong \mathbb{C}[x,y]/(xy,(x+y)^4)$. The kernel of multiplication by $y^2-x^2$ is the submodule 
\[\left( x^2+y^2 \right)\cong \mathbb{C}\left\lbrace x^2+y^2,x^3,y^3,x^4 \right\rbrace \cong \chi_2\chi_3 \oplus V \oplus \chi_4.\]
The cokernel of multiplication by $x^2-y^2$ is 
\[\mathbb{C}[x,y]/(xy,x^2-y^2) \otimes \chi_2\chi_3\chi_4\cong \chi_2\chi_3\chi_4\oplus V \oplus \chi_4.\]
Observe that the middle cohomology is the sum of the above kernel and cokernel. Then, none of the cohomologies of \eqref{eq:koszul1Dim} have invariant summands.
\end{proof}

Next, we focus on the zero-dimensional loci and compare them to the positive-dimensional ones. We start from this computation:

\begin{lemma}
  We have isomorphisms in \(\mathcal{D}(BG)\):
  \begin{align*}
  \mathrm{RHom}(\mathcal{O}_0,\Phi_{1}(\mathcal{O}_0)) &\cong
    \chi_2\chi_4[0]\oplus \left( \chi_2\chi_4 \right)^{\oplus 2}[-1]\oplus
    \chi_2\chi_4[-2]\\
    \mathrm{RHom}(\mathcal{O}_0,\Phi_{2}(\mathcal{O}_0)) &\cong
    \chi_3[0]\oplus \left( \chi_3\oplus\chi_2\chi_4 \right)[-1]\oplus
    \chi_2\chi_4[-2]\\
    \mathrm{RHom}(\mathcal{O}_0,\Phi_{3}(\mathcal{O}_0))
    &\cong \chi_4[0]\oplus (\chi_4\oplus\chi_2\chi_4)[-1]\oplus
    \chi_2\chi_4[-2] \\
    \mathrm{RHom}(\mathcal{O}_0,\Phi_{4}(\mathcal{O}_0))
    &\cong \chi_3\chi_4[0]\oplus (\chi_3\chi_4\oplus\chi_2\chi_4)[-1]\oplus
    \chi_2\chi_4[-2].
  \end{align*}
  \label{lem:ext-groups}
\end{lemma}

\begin{proof}
  We again only compute the case of \(\Phi_3\). It is computationally simpler to use
 Serre duality first:
  \begin{align*}
    \mathrm{Ext}^\ast_{[X/G]}(\mathcal{O}_0,\Phi_{3}(\mathcal{O}_0))
    &\cong 
    \mathrm{Ext}^{2-\ast}_{[X/G]}(\Phi_{3}(\mathcal{O}_0),\mathcal{O}_0\otimes\BigWedge^2V)^\vee.
  \end{align*}
  We resolve \(\Phi_{3}(\mathcal{O}_0)\) (computed in Proposition \ref{prop:image of points234}) by the Koszul complex noting
  that \(xy\) has weight \(\chi_2\):
  \begin{align*}
    \mathrm{Ext}^\ast_{[X/G]}(\Phi_{3}(\mathcal{O}_0),\mathcal{O}_0\otimes\BigWedge^2V)
    &\cong
    \left(H^\ast(\mathcal{O}_0\chi_2\chi_4\to\mathcal{O}_0\chi_2\chi_4\oplus\mathcal{O}_0\chi_4\to
    \mathcal{O}_0\chi_4)^\vee\right)^G
  \end{align*}
  and hence the statement.
\end{proof}

Finally, we study Hom spaces between zero-dimensional loci:

\begin{lemma} 
\label{lem:orthogonalitiesDim0}
Let $\rho, \sigma$ be irreducible representations of $G$. Then, the position $(\rho,\sigma)$ is marked by a \checkmark in the following table if and only if
\[ \mathrm{Ext}^*_{[X/G]}(\mathcal{O}_0 \otimes \rho,\mathcal{O}_0\otimes \sigma )=0. \]
\begin{center}

\begin{tabular}{|l||*{10}{c|}}\hline
\backslashbox{$\rho$}{$\sigma$}
&$\mathbf{1}$&$\chi_2$&$\chi_3$&$\chi_4$&$\chi_2\chi_3$&$\chi_2\chi_4$&$\chi_3\chi_4$&$\chi_2\chi_3\chi_4$&$V$&$V^\vee$
\\\hline\hline
$\mathbf{1}$&&\checkmark&\checkmark&\checkmark&\checkmark&&\checkmark&\checkmark&\checkmark&\\\hline
$\chi_2$&\checkmark&&\checkmark&&\checkmark&\checkmark&\checkmark&\checkmark&&\checkmark\\\hline
$\chi_3$&\checkmark&\checkmark&&\checkmark&\checkmark&\checkmark&\checkmark&&\checkmark&\\\hline
$\chi_4$&\checkmark&&\checkmark&&\checkmark&\checkmark&\checkmark&\checkmark&\checkmark&\\\hline
$\chi_2\chi_3$&\checkmark&\checkmark&\checkmark&\checkmark&&\checkmark&&\checkmark&&\checkmark\\\hline
$\chi_2\chi_4$&&\checkmark&\checkmark&\checkmark&\checkmark&&\checkmark&\checkmark&&\checkmark\\\hline
$\chi_3\chi_4$&\checkmark&\checkmark&\checkmark&\checkmark&&\checkmark&&\checkmark&\checkmark&\\\hline
$\chi_2\chi_3\chi_4$&\checkmark&\checkmark&&\checkmark&\checkmark&\checkmark&\checkmark&&&\checkmark\\\hline
$V$&&\checkmark&&&\checkmark&\checkmark&&\checkmark&&\\\hline
$V^\vee$&\checkmark&&\checkmark&\checkmark&&&\checkmark&&&\\\hline
\end{tabular}
\end{center}  
\end{lemma}

\begin{proof}
Twisting the \eqref{eq:EndomorphismsO_0} we obtain
 \begin{equation*}
    \mathrm{Ext}^\ast_{[X/G]}(\mathcal{O}_0\otimes \rho,\mathcal{O}_0\otimes \sigma) = \left[(\mathbf{1}[0]
    \oplus V[-1]\oplus \BigWedge^2V[-2])\otimes (\rho^\vee\otimes \sigma)[0]\right]^G.
 \end{equation*}
and one checks which choices of $(\rho,\sigma)$ produce invariant summands, using Proposition \ref{prop:tensorsquareV}. 
\end{proof}

It follows from Lemma \ref{lem:ext-groups} that \(\mathcal{O}_0\), \( \mathcal{O}_0\otimes \chi_2 \), and
\(\mathcal{O}_0\otimes V^\vee\) are in the left orthogonal of the images of the
positive-dimensional loci. In particular, the sheaf \( M \coloneqq\mathbb{C}[x,y]/(x^2+y^2,x^2-y^2)\) is also in the left orthogonal, since as a representation 
\begin{equation}
\label{eq:M_representation}
M\cong \mathbf{1} \oplus V^\vee \oplus \chi_2.
\end{equation}

\begin{proposition}
\label{prop:exc_collection}
The collection
\begin{equation}
\label{eq:exc_collection}
\left(\mathcal{O}_0 \otimes \chi_2,\,
M,\,
\mathcal{O}_0 \otimes \chi_2\chi_3,\,
\mathcal{O}_0 \otimes \chi_2\chi_3\chi_4,\,
\mathcal{O}_0 \otimes V,\,
\mathcal{O}_0  \right)
\end{equation}
is an exceptional collection in $\mathcal{D}[X/G]$.
\end{proposition}

\begin{proof}
First, we observe that all objects are exceptional. The exceptionality of the twists of $\mathcal{O}_0$ by the characters is immediate. For the twist by $V$, it follows from Proposition \ref{prop:tensorsquareV}. We check $M$ now. Its endomorphism algebra is computed by resolving $M$ by a Koszul complex and computing the cohomology of
\[M \xrightarrow{0} 
M \chi_2\chi_3
\oplus  
M \chi_2\chi_3\chi_4
\xrightarrow{0} M \chi_4\]
($\chi_2\chi_3$ and $\chi_2\chi_3\chi_4$ are the weights of $x^2+y^2$, $x^2-y^2$ respectively), so by \eqref{eq:M_representation} it only has one invariant summand in degree 0, i.e. $M$ is exceptional.

We now check orthogonalities. All the conditions not involving $M$ follow from Lemma \ref{lem:orthogonalitiesDim0}, and all the ones involving $M$ are similar: we only compute $\mathrm{Ext}^*(M,\mathcal{O}_0\otimes \chi_2)^G=0$ as an example. Once again, resolve $M$ with a Koszul complex and consider the cohomology of 
\[ \mathcal{O}_0\chi_2 \to \mathcal{O}_0\chi_3 \oplus \mathcal{O}_0\chi_3\chi_4 \to \mathcal{O}_0\chi_2\chi_4,  \]
which has no invariant summands.  
\end{proof}

For $i=1,...,4$, let $\mathcal C_i$ denote the image of $\Phi_i\colon \mathcal{D}(\overline{X_i}) \to \mathcal{D}[X/G]$. 
For $i\geq 5$, define $\mathcal C_i$ to be the image of the twisted functors
 \begin{align*}
    \Phi'_5 &= \Phi_5(-)\otimes \chi_2 
    & \Phi'_8 &= \Phi_8(-)\otimes \chi_2\chi_3\chi_4 \\
    \Phi'_6 &= \Phi_6(-)\otimes M 
    & \Phi'_9 &= \Phi_9(-)\otimes V\\
   \Phi'_7 &= \Phi_7(-)\otimes \chi_2\chi_3 
    & \Phi'_{10} &=\Phi_{10}(-).
  \end{align*}

\begin{theorem}
\label{thm:msod_A2}
There is an orbifold semiorthogonal decomposition
\begin{equation}
\label{eq:msod_A2}
\mathcal{D}[X/G]=\left\langle \mathcal C_1,\mathcal C_2,\mathcal C_3,\mathcal C_4,\mathcal C_5,\mathcal C_6,\mathcal C_7,\mathcal C_8,\mathcal C_9,\mathcal C_{10} \right\rangle.
\end{equation}
\end{theorem}

\begin{proof}
For all $i=1,...,10$, the component $\mathcal C_i$ is equivalent to $\mathcal D(\overline{X_i})$. This is the content of Proposition \ref{prop:ffcomponents} (the generators of $\mathcal C_5,...,\mathcal C_{10}$ are still exceptional objects by Proposition \ref{prop:exc_collection}).

Proposition \ref{prop:exc_collection} also shows that $\mathcal C_5,...,\mathcal C_{10}$  are semiorthogonal. Moreover, it immediately follows from Lemma \ref{lem:ext-groups} that $\mathrm{Ext}^\ast_{[X/G]} (\mathcal C_j,\mathcal C_i)=0$ for $i=1,2,3,4$ and $j=5,...,10$. 
The first statement of Lemma \ref{lem:orthogDim1and2} implies $\mathrm{Ext}^\ast_{[X/G]}(\mathcal C_i,\mathcal C_1)=0$ for $i=2,3,4$. Now let $i,j$ be distinct elements of $\left\lbrace 2,3,4 \right\rbrace$. Then, the Fourier-Mukai kernels $Z^i$ and $Z^j$ only intersect at the origin, therefore the second statement of Lemma \ref{lem:orthogDim1and2} implies the vanishing $\mathrm{Ext}^\ast_{[X/G]}(\mathcal C_i,\mathcal C_j)=0$.

It remains to show that the subcategory $\mathcal{T}\coloneqq\left\langle \mathcal C_1,...,\mathcal C_{10} \right\rangle$ coincides with $\mathcal{D}[X/G]$. A special case of the spanning classes of \cite[Prop. 2.1]{L-P-ffcriterion} is given by 
\[\Omega\coloneqq \left\lbrace \mathcal{O}_p\otimes \rho \;\vert\; p\in X, \rho \in \mathrm{Irrep}(G_p) \right\rbrace\]
(here $G_p$ denotes the stabilizer of $p$ in $G$). 
Then, it suffices to show that $\mathcal T$ contains $\Omega$.
Indeed, if $\Omega \subset \mathcal T$, every object $E\in \mathcal{T}^\perp$ satisfies 
$\mathrm{Hom}(\omega, E)=0$ for all $\omega\in \Omega$, which implies $E=0$ by definition of a spanning class, and therefore $\mathcal{T}^\perp=\left\lbrace 0 \right\rbrace$. 

Recall from the proof of Proposition \ref{prop:image of points1} that $G_p$ can only be $G$ (if $p=0$), $\mu_2$ (if $p\in X_i$ for some $i=2,3,4$) or trivial.
If $p=0$, then $G_p\cong G$, and we recover all irreducible representations as follows: from the exceptional collection \eqref{eq:exc_collection} we get 
\[ \chi_2, \chi_2\chi_3, \chi_2\chi_3\chi_4, V, \mathbf{1}, \mbox{ and } V^\vee\]
(the last one from the composition series of $M$, whose factors are $(\chi_2,V^\vee,\mathbf{1})$). With these, and using Corollary \ref{cor:rep-decomp-origins}, we obtain $\chi_3,\chi_4,\chi_3\chi_4$, and $\chi_2\chi_4$. 


Suppose $p\neq 0$ has a $\mu_2$ stabilizer. On the one hand, the pull-back along $X_i \to \overline{X_i}$ contains $\mathcal{O}_p$ as a direct summand as in the proof of Proposition \ref{prop:image of points234}. On the other hand, the pull-back $\pi^*(\mathcal{O}_{\pi(p)})$ is non-reduced of length 2 at every point of its support by the proof of Proposition \ref{prop:image of points1}. Then, the non-trivial irreducible representation is obtained as the kernel of the map
 \[ \pi^*(\mathcal{O}_{\pi(p)}) \to \mathcal{O}_p. \]
 
 If $G_p=0$, then $\Phi_{1}(\mathcal{O}_{\pi(p)})$ contains $\mathcal{O}_p$ as a direct summand (by Proposition \ref{prop:image of points1}), hence $\mathcal{O}_p\in \mathcal{T}$. \qedhere
\end{proof}

\section{Orbifold Semiorthogonal Decomposition for Type (C)}
\label{sec:type-c}

\subsection{Setup} The construction of the orbifold semiorthogonal decomposition for Type (C) (see Theorem \ref{thm:zero-fixed}) will heavily rely on the local analysis carried out in \S 3. We recall here the construction of Type (C) quotients, following \cite{ga-smooth2}. Let \(\Lambda = \mathbb{Z}[\xi]\) be the Gaussian lattice and \(E\cong
\mathbb{C}/\Lambda\) the corresponding elliptic curve with
\(\mu_4\) automorphism group. The natural action \eqref{eq:NaturalAction} of $G\coloneqq G(4,2,2)$ on $\mathbb{A}^2\simeq \mathbb{C}^2$ preserves the lattice $\Lambda^{\oplus 2}$ and descends to an action on $B\coloneqq E^2$. However, this is \textit{not} the correct action to consider, since the quotient $B/G$ is not smooth\footnote{Consider the point \( (e,t_0) \in B\), where \(e\in E\) is the origin and \(t_0 \in E\) is the only non-trivial \(\xi\)-invariant point. Let $G$ act through the natural action on the universal cover \eqref{eq:NaturalAction}. A direct computation shows that the stabilizer of $(e,t_0)$ is the subgroup $K < G$ (since the coordinates cannot be permuted). $K$ is generated by $(\xi,\xi,1),(-1,1,1)$.
One checks directly that the only elements of $K$ acting by pseudoreflections are $(-1,1,1)$ and $(1,-1,1)$. In particular, $(\xi,\xi,1)$ is not generated by pseudoreflections in $K$ (see \cite[Proposition 3.4]{ga-smooth1}).}. One can instead define another action of $G$ on $E^2$ as follows. Consider another copy of the Abelian surface  \(A \coloneqq E^2\) and an isogeny \(\pi\colon B\to A\) defined by the matrix
\[
  \nu = \begin{pmatrix}
    1 & -1 \\ 0 & \xi-1
  \end{pmatrix}.
\]
Let \(t_0\) denote the only non-trivial \(\xi\)-invariant element of \(E\), then
the kernel of \(\pi\) is
\[
  \Delta \coloneqq \ker(\pi) = \left\langle(t_0,t_0)\right\rangle,
\]
and we have \(B/\Delta\cong A\). Moreover, $G$ acts trivially on $\Delta$, so we have an action of \(\Delta\times G\) on \(B\). This action descends to an action of \(G\) on \(A\) and induces a natural isomorphism \(B/\Delta\times G\cong
A/G\). Explicitly, every element $g$ of $G$ acts on $A$ via the matrix $\nu \lambda \nu^{-1}$ where $\lambda$ is the matrix expression of $g$ acting on $\mathbb{A}^2$ via \eqref{eq:NaturalAction}. In particular, the generators \( (-1,1,1), (-\xi,\xi,1), (1,1,\sigma)\) act on
\(A\), respectively, via the matrices:
\begin{equation}
    \label{eq:matricesABC}
  \alpha = \begin{pmatrix}
    -1 & 1+\xi \\ 0 & 1
  \end{pmatrix},\ \beta = \begin{pmatrix}
    -\xi & \xi-1 \\ 0 & \xi 
  \end{pmatrix},\ \gamma = \begin{pmatrix}
    -1 & 0 \\ \xi-1 & 1
  \end{pmatrix}.
  \end{equation}

\subsection{Fixed loci and stabilizers}
\label{ssec:typeC_fixed_loci}

We use the numbering of Proposition \ref{prop:conjclassesg422} for the conjugacy classes and the centralizers of $G$, and we list representatives of conjugacy classes and fixed loci. By $E[n]$ we denote the set \(\{x\in E\mid nx = 0\}\) of $n$-torsion points.
The conjugacy classes are listed in Proposition \ref{prop:ClassesG}. For convenience we recall them here and write every element in terms of the matrices \eqref{eq:matricesABC}. 
\begin{equation*}
  \begin{aligned}
    D_1 &= \{\Id\} & 
    D_2 &= \{\gamma, -\gamma \} \\
    D_3 &= \{\alpha, -\alpha\} & 
    D_4 &= \{\beta\gamma, -\beta\gamma\} \\
    D_5 &= \{\beta^2\}& 
    D_6 &= \{\alpha\beta\}\\
    D_7 &= \{\alpha\beta\gamma, -\alpha\beta\gamma\} &
    D_8 &= \{-\alpha\beta\} \\
    D_9 &= \{\alpha\gamma,-\alpha\gamma\} & 
    D_{10} &= \{\beta,-\beta\}.
    \end{aligned}
  \end{equation*}

We compute the corresponding fixed loci (we pick the first element listed in every conjugacy class as its representative). For example, 
\begin{align*}
A_3 = A^\alpha  & =  \{(x,y)\in A \mid (x,y)=\alpha(x,y)=(-x +(1+\xi)y, y)\} \\
 & = \{(x,y)\in A \mid 2x=(1+\xi)y \} \\
 & = \{(x,y)\in A \mid x=\frac{1}{2}(1+\xi)y + E[2] \} \\
 & \simeq E[2] \times E.
\end{align*}
Similarly:
\begin{equation*}
\begin{aligned}
  A_1 &= A\\
  A_{2} &= \{(x,y)\in A\mid x = \xi x \}\cong \{e,t_0\}\times  E\\
  A_{3} &= \{(x,y)\in A \mid 2x=(1+\xi)y\}\cong E[2]\times E\\
  A_{4} &= \{(x,y)\in A \mid x=\xi y\}\cong E\\
  A_{5} &= E[2]^2\\
  A_{6} &= ... = A_{10}= \{e,t_0\}^2,
  \end{aligned}
\end{equation*}
where $\{e,t_0\}^2$ is a shorthand for $\{ (e,e),(e,t_0),(t_0,e),(t_0,t_0) \}$.
We point out that the only points of $A$ with non-trivial stabilizers are:
\begin{center}
\begin{tabular}{l | c}
Points & Stabilizer \\\hline
$\{e,t_0\}^2$ & $G$ \\
$E[2]^2\setminus \{e,t_0\}^2$ & $\mu_2\times \mu_2$\\
$(A^2\cup A^3\cup A^4)\setminus E[2]^2$ & $\mu_2$.\\
\end{tabular}
\end{center}

Finally, we compute the quotients $\overline{A^i}$. Observe that the group $G$ acts on $E[2]^2$ with 7 orbits. Four orbits are given by elements of $\{e,t_0\}^2$, which are fixed by $G$. The remaining three orbits (let $a\coloneqq \frac{1}{2}$ and $b\coloneqq \frac{\xi}{2}$) are the sets:
\begin{align*}
    \{e,t_0\}\times \{a,b\}\\
    \{a,b\}\times \{e,t_0\}\\
    \{a,b\}^2.
\end{align*}

This shows immediately that $\overline{A^5}= E[2]^2/G $ consists of the 7 orbits above, and that the quotients 
$$ \overline{A^{6}} = ... = \overline{A^{10}}= \{e,t_0\}^2. $$

We have $\overline{A^{1}}=\mathbb P^2$ by \cite[Theorem 1.1]{ga-smooth2}. It is left to compute the quotients of the one-dimensional loci. The computations are similar for $i=2,3,4$. For example, we have 
$$ \overline{A^3}= \{(x,y)\in A \mid x=t_0y+E[2]\}/\pair{\alpha,\beta}, $$
and one sees that, for any $t\in E[2]$, we have 
\[ \beta \begin{pmatrix}
    t_0y +t \\ y
  \end{pmatrix} = \begin{pmatrix}
    t_0(\xi y) -\xi t \\ \xi y
  \end{pmatrix} \]
In other words, $\beta$ acts by identifying two copies of $E$ (and by negation on each of them) and by the order 4 automorphism on the other two copies: 
$$\overline{A^3}\simeq \mathbb{P}^1 \sqcup \mathbb{P}^1 \sqcup\mathbb{P}^1.$$
Similarly, one sees $\overline{A^2}=\mathbb{P}^1 \sqcup \mathbb{P}^1$ and  $\overline{A^4}=\mathbb{P}^1$.

\subsection{The orbifold semiorthogonal decomposition}

We construct Fourier-Mukai functors as in \S \ref{ssec:FM_functors_A2}, by taking the structure sheaves of (equivariantized) graphs of the inclusions of the fixed loci in $A^i\times A$. 
We denote by $\Phi_i\colon \mathcal{D}(\overline{A^i}) \to \mathcal{D}[A/G]$ the corresponding Fourier-Mukai functors. 
The fixed loci $A^i$ with $i\geq 6$ only contain points that are fixed by $G$. Locally at each of these points, the action of $G$  coincides with that described in Section \ref{sec:msodA2}, so we can modify the $\Phi_i$ by twisting 
 \begin{align*}
    \Phi'_6 &= \Phi_6(-)\otimes M 
    & \Phi'_9 &= \Phi_9(-)\otimes V\\
    \Phi'_7 &= \Phi_7(-)\otimes \chi_2\chi_3 
    & \Phi'_{10} &=\Phi_{10}(-),\\
    \Phi'_8 &= \Phi_8(-)\otimes \chi_2\chi_3\chi_4
    & 
  \end{align*}
  where $M$ is the torsion sheaf supported at $\{e,t_0\}^2$ whose fiber at every point is locally isomorphic to the module $\mathbb{C}[x,y]/(x^2+y^2,x^2-y^2)$ (see \S \ref{ssec:assembling_MSOD_A2}).
  
We then modify $\Phi_5$ by only twisting by a character on points with stabilizer $G$, and defining for $X\in \mathcal{D}(\overline{A^5})$:
\begin{equation*}
\Phi_5'(X)\coloneqq \begin{cases}
\Phi_5(X)\otimes \chi_2 \ \ & \mbox{ if } \ \ \mathrm{Supp}(X)\subset \{ e,t_0 \}^2 \\
\Phi_5(X) \ \ & \mbox{ if } \ \ \mathrm{Supp}(X)\subset \overline{A^5}\setminus\{ e,t_0 \}^2 \end{cases}
\end{equation*}  
and extending additively to $\mathcal{D}(\overline{A^5})$. Then we have:

\begin{theorem}
\label{thm:msod_typeC}
In the case of a smooth quotient $A/G$ of type (C), the functors $\Phi_1,...,\Phi_4,\Phi_5',...,\Phi_{10}'$ give rise to an orbifold semiorthogonal decomposition
\begin{equation}
\label{eq:msod_typeC}
 \mathcal{D}[A/G]=\left\langle \mathcal{D}(\overline{A^1}),...,\mathcal{D}(\overline{A^{10}})\right\rangle. 
\end{equation}
\end{theorem}

\begin{proof}
Everything can be checked locally around each point $p\in A$, so it suffices to perform a local computation based on stabilizer types: \(G\), \(\mu_2\times\mu_2\), \(\mu_2\).  

For example, suppose \(p\in A\) satisfies $\mathrm{Stab}_G(p)=G$ and \(x\in \overline{A}_i\) is such that \(\Phi_i(\mathcal{O}_{x})\) is supported at \(p\). Pick a \(G\)-invariant affine open subset \(\mathrm{Spec}(R)\) containing \(p\) and let \(\mathfrak{m}\) be the corresponding maximal ideal. Then
\[
	\mathrm{Ext}^*_{[A/G]}(\Phi_i(\mathcal{O}_{x}),\Phi_i(\mathcal{O}_{x}))
	\cong \mathrm{Ext}^*_R(\Phi_i(\mathcal{O}_{x}),\Phi_i(\mathcal{O}_{x}))^G
\]
Since \(\Phi_i(\mathcal{O}_{x})\) are finite type and supported at \(\mathfrak{m}\), we can pass to the completion
\[
	\mathrm{Ext}^*_R(\Phi_i(\mathcal{O}_{x}),\Phi_i(\mathcal{O}_{x}))^G
	\cong \mathrm{Ext}^*_{\hat{R}}(\Phi_i(\mathcal{O}_{x}),\Phi_i(\mathcal{O}_{x}))^G.
\]
But \(\hat{R}\cong \mathbb{C}[[a,b]]\) since \(A\) is smooth. Since we have the usual identification 
\[
	\mathbb{C}\{a,b\}\cong \mathfrak{m}/\mathfrak{m}^2\cong T_pA,
\]
it follows that \(\mathbb{C}\{a,b\}\) is an irreducible representation of \(G\) and hence isomorphic to \(V\) or \(V^\vee\). Hence, if the stabilizer is isomorphic to \(G\) itself the decomposition \eqref{eq:msod_typeC} coincides locally with the one in  \eqref{eq:msod_A2}, and the statement follows from Theorem \ref{thm:msod_A2}.


Suppose now that $p$ has a $\mu_2\times\mu_2$ stabilizer: the local model around $p$ is $\mathbb{A}^2$ where each copy of $\mu_2$ acts by negating a coordinate, and it has an orbifold semiorthogonal decomposition 
\begin{equation}
\label{eq:sod_m2xm2}
 \mathcal{D}[\mathbb{A}^2/\mu_2\times\mu_2] = \pair{ 
\mathcal{D}(\overline{\mathbb{A}^2}),
\mathcal{D}(\overline{\mathbb{A}^1}),
\mathcal{D}(\overline{\mathbb{A}^1}),
\mathcal{D}(\mathrm{pt})}. 
\end{equation}

When restricted to some neighborhood of $p$, the categories in \eqref{eq:msod_typeC} yield the semiorthogonal decomposition \eqref{eq:sod_m2xm2}. For example, let $p=(e,a)$. We have
\[p \in A^1 \cap A^2 \cap \alpha_*A^2 \cap A^5.\]
Then we have local isomorphisms
\begin{align*}
\Phi_1(\mathcal{O}_p) & \simeq \mathbb{C}[u,v]/(u^2,v^2)\\
\Phi_2(\mathcal{O}_p) & \simeq \mathbb{C}[u,v]/(u^2,uv,v^2)\\
\Phi_5(\mathcal{O}_p) & \simeq \mathbb{C}[u,v]/(u,v),
\end{align*}
which satisfy the correct orthogonalities by \eqref{eq:sod_m2xm2} and generate $\mathcal{D}[A/G]$ at $p$.

The case of stabilizer of type \(\mu_2\) is similar.
\end{proof}

\begin{corollary}
The category $\mathcal{D}[A/G]$ admits a full exceptional collection of length 42. 
\end{corollary}

\begin{proof}
In fact, each of the pieces $\mathcal{D}(\overline{A^i})$ appearing in \eqref{eq:msod_typeC} admits a full exceptional collection: the $\overline{A^i}$ are computed in \S \ref{ssec:typeC_fixed_loci} and are unions of projective spaces and points. In particular, the two-dimensional quotient contributes 3 exceptional objects, the one-dimensional quotients (6 copies of the projective line) contribute 12 exceptional objects, and there are $7+5\times 4$ irreducible components in the zero-dimensional quotients.
\end{proof}

\section{Orbifold semiorthogonal decompositions for \texorpdfstring{\(\DD[A/G]\)}{D[A/G]}}
\label{sec:msod-abelian}

In this Section we prove Theorem \ref{thm:main-msod}. We will break the proof up into three parts. In \S \ref{ssec:zero-dim}, we prove
the case of zero-dimensional fixed loci where \(G\) acts by group automorphisms.
In \S \ref{ssec:pos-dim}, we prove the case of positive-dimensional fixed loci
where \(G\) acts by group automorphisms. In \S \ref{ssec:trans}, we use the
results of \S \ref{ssec:zero-dim} and \S \ref{ssec:pos-dim} to prove the case
of arbitrary \(G = T \rtimes H\), with $T$ a normal subgroup of translations and $H$ the subgroup of group automorphisms. 

\subsection{Zero-dimensional fixed loci}
\label{ssec:zero-dim}

Suppose \(G\) acts on \(A\) by group automorphisms. We assume \(G\) acts irreducibly on
\(T_e(A)\) and that \(\dim(A^G)=0\). We will need the following result from \cite{L-P-divisor}. We recall it here and show how it applies to the current case.

\begin{theorem}[{\cite[Theorem 1.2.1]{L-P-divisor}}]\label{thm:LPDivisor}
Suppose $G$ is a finite group acting effectively on a smooth quasi-projective variety $X$, and that $\mathcal D[X/G]$ admits an orbifold semiorthogonal decomposition. Let $H \subset X$ be a smooth $G$-invariant divisor. If $H$ satisfies the generality assumption $(\ast)$ below, then $\mathcal D[H/G]$ also admits an orbifold semiorthogonal decomposition.  
\end{theorem}

For each conjugacy class $\lambda$ 
let $W_\lambda$ be the quotient of $C(\lambda)$ acting effectively on $X_\lambda$ and write $X_\lambda^{\fr}$ for the locus where the action is free. Then the assumption is

\smallskip

\begin{displayquote}
$(\ast)$ for every $\lambda$, $H$ does not contain $X_\lambda$ and $H\cap X_\lambda^{\fr}$ is dense in $H\cap X_\lambda$.
\end{displayquote}

\smallskip

Recall quotients of Type (B) from Theorem \ref{thm:ALAQClassification}: the group $G\simeq S_{n+1}$ acts permuting the variables on the product $X = E^{n+1}$ of an elliptic curve $E$ with itself, and $A\subset X$ is the smooth $S_{n+1}$-invariant divisor defined by \(x_1+\cdots +  x_{n+1} = 0\).

\begin{lemma}\label{lem:MSODTypeB}
    If $A$ is as in Type (B) of Theorem \ref{thm:ALAQClassification}, then $\mathcal D[A/S_{n+1}]$ admits an orbifold semiorthogonal decomposition. 
\end{lemma}

\begin{proof}
We check that Theorem \ref{thm:LPDivisor} applies. First of all, $\mathcal D [X/S_{n+1}]$ admits an orbifold semiorthogonal decomposition by \cite[Theorem B]{pvdb-equivariant}. Next, we check condition $(\ast)$. 

Write \(\lambda = (1^{r_1})(2^{r_2})\cdots(k^{r_k})\). Then we can identify \(W_\lambda = C(\lambda)/\left\langle \lambda \right\rangle\cong S_{r_1}\times S_{r_2}\times\cdots\times S_{r_k}\), and 
\[
X_\lambda \cong E^{r_1}\times E^{r_2}\times\cdots E^{r_k}.
\]
Thus, \(A\cap X_\lambda\) is identified with the set of points \(e\in X\) such that
\begin{equation}
e_1+\cdots +e_{r_1}+2e_{r_1+1}+\cdots + 2e_{r_1+r_2}+\cdots + k e_{r_1+\cdots + r_{k-1}+1}+\cdots + ke_{r_1+\cdots + r_k} = 0.
\label{eq:AcapXlambda}    
\end{equation}

On the other hand, \(A\cap X_\lambda^{\fr}\) is the subset of \(A\cap X_\lambda\) lying in the complement of the diagonals in the factors of $X_\lambda$. More explicitly, $e \in A\cap X_\lambda^{\fr}$ if
 \begin{equation*}
     \begin{cases}
         \mbox{all of the } e_1,\ldots,e_{r_1} \mbox{ are distinct, and}\\
         \mbox{all of the }e_{r_1+1},\ldots,e_{r_1 + r_2} \mbox{ are distinct, and}\\
         \ldots\\
         \mbox{all of the }e_{r_1+ \cdots+ r_{k-1}+1},\ldots, e_{r_1+\cdots + r_k} \mbox{ are distinct.}\\
     \end{cases}
 \end{equation*}
 This is clearly open in $A\cap X_\lambda$, and any non-free point of $A\cap X_\lambda$ can be deformed to one of $A\cap X_\lambda^{\fr}$ while satisfying equation \eqref{eq:AcapXlambda}. In other words, $A\cap X_\lambda^{\fr}$ is dense in $A\cap X_\lambda$. We conclude that \(\mathcal D[A/S_{n+1}]\) admits an orbifold semiorthogonal decomposition.
\end{proof}

We can now prove the following theorem. 

\begin{theorem}
  Suppose \(G\) is a finite group of automorphisms of an Abelian variety \(A\)
  that fix the identity. Assume  that \(G\) acts irreducibly on
  \(T_e(A)\) and that \(\dim(A^G)=0\). Then \(\DD[A/G]\) has an orbifold semiorthogonal decomposition.
  \label{thm:fixing-origin}
\end{theorem}

\begin{proof}
Under the assumptions, the classification of Theorem \ref{thm:ALAQClassification} holds, so we need to cover three cases. 

  The case \(G\cong (\mu_k)^n\rtimes S_n\) with \(k\neq 1\) is covered in
  \cite[Section 4.3]{pvdb-equivariant}. 
Type (B) is covered in Lemma \ref{lem:MSODTypeB}.
  Lastly, if \(A\) is a surface of Type (C), this is Theorem
  \ref{thm:msod_typeC}.
\end{proof}

\subsection{Positive-dimensional fixed loci}
\label{ssec:pos-dim}

Suppose now that \(\dim(A^G)>0\). We will reduce this case to the setting of \S \ref{ssec:zero-dim}. Intuitively, we would like to decompose $A$ as a product of abelian varieties $P_G \times A_0$, where \(A_0\subset A^G\) is the connected
component containing the identity and $G$ acts on $P_G$ with $\dim (P_G)^G=0$. Then, we could apply Theorem \ref{thm:fixing-origin} to $[P_G/G]$ and combine it with the trivial action of $G$ on $A_0$ to prove the theorem for $A$. 

This is subtle in general, since $A$ is only isogenous to such a product. More precisely, there is an isogeny 
\[
  P_G\times A_0 \to A
\]
with kernel \(\Delta = P_G\cap A_0\)\footnote{Here, $P_G$ is taken to be a complementary abelian subvariety of $A_0$. We refer the reader to \cite[Section 5.3]{BL04_AbelianVarieties} for generalities on complementary abelian varieties, and to \cite[\S 2.3]{ga-smooth1} for a construction similar to the one at hand. The notation $P_G$ is reminiscent of the fact that if $A$ is the Jacobian of a curve $X$ and $G$ acts on $X$, then $P_G$ is the Prym variety of the morphism $X \to X/G$.}.
 Moreover, \(\Delta\) acts freely on \(A_0\) and there is an isomorphism
\[
  P_G\times A_0/\Delta\cong A.
\]
Set \(\tilde{A} = P_G\times A_0\) for
convenience. Composing the above isomorphism with pushforward along $[\tilde{A}/\Delta] \to \Tilde{A}/\Delta \cong A$ induces an equivalence:
\[
  \DD(\tilde{A})^\Delta\cong \DD(A).
\]
Moreover, since \(\Delta\subset A^G\), we see that \(\Delta\) and \(G\) commute
with each other inside \(\mathrm{Aut}(P_G\times A_0)\). So we can consider 
\(  [\tilde{A}/G\times \Delta]\) and the corresponding derived equivalence:
\[
  \DD[\tilde{A}/G]^\Delta\cong \DD[\tilde{A}/G\times \Delta]\cong \DD[A/G].
\]
(Note, $\Delta$ acts diagonally, on both \(P_G\) and \(A_0\)).

Our goal is then to obtain an orbifold semiorthogonal decomposition for $[\Tilde{A}/G]$ and to modify it suitably to make it $\Delta$-equivariant and descend it to $[A/G]$. 

Observe that the quotient stack \([\tilde{A}/G]\) possesses an orbifold semiorthogonal
decomposition. In fact, we may write
\[ \Tilde{A}/G = (P_G\times A_0)/G = (P_G/G) \times A_0 \]
since $G$ acts trivially on $A_0$. By construction, $\dim (P_G)^G =0$, so $[P_G/G]$ admits an orbifold semiorthogonal decomposition by Theorem \ref{thm:fixing-origin}. Then, so does $[\Tilde{A}/G]$ using \cite[Lemma 2.4.1]{L-P-divisor}.

Consider the functors defining the decomposition:
\[
  \Phi_\lambda\colon \DD(P_G^\lambda/C_G(\lambda)\times A_0) \hookrightarrow
  \DD[\tilde{A}/G].
\]
The functor \(\Phi_\lambda\) corresponds to a \(G\)-equivariant kernel
\(\KK_\lambda\) on \(\tilde{A}^\lambda\times \tilde{A}\), where \(G\) acts on
the right. 

Define the subscheme of \( (P_G^\lambda\times A_0)\times \tilde{A}\)
\[
	Z_\lambda^\Delta = \bigcup_{(g,\delta)\in G\times  \Delta} (g,\delta)^\ast\Gamma
\]
with its reduced scheme structure, where \(\Gamma\) is the graph of the inclusion. Equivalently, if \(Z_\lambda\) is the Fourier-Mukai kernel defining \(\Phi_\lambda\), then
\[
	Z_\lambda^\Delta = \bigcup_{\delta\in\Delta}\delta^\ast Z_\lambda.
\]

\begin{lemma}
	For each distinct \(\delta, \tau\in \Delta\) the subschemes \(\delta^*Z_\lambda, \tau^*Z_\lambda\) do not intersect.
\label{lem:delta-equivariant}
\end{lemma}

\begin{proof}
	It suffices to prove the claim when \(\tau\) is the identity. Suppose \( (x,g(x)) = (x,\delta(h(x)))\) for some \(g,h\in G\). Since the action of \(G\) commutes with \(\Delta\) we have \(x = \delta(g^{-1}h(x))\). But \(\Delta\) acts freely on the quotient space \(\tilde{A}/G\) as it acts freely on \(A_0\), a contradiction.
\end{proof}

\begin{theorem}
  For each \(\lambda\in G/\sim\), the functors \(\Phi_\lambda\) are
  fully-faithful and there exists a total order on \(G/\sim\) such that the
  \(\Phi_\lambda\) give rise to an orbifold semiorthogonal decomposition of \(\DD[A/G]\).
  \label{thm:pos-dim-fixed}
\end{theorem}

\begin{proof}
  We show the functors defining orbifold semiorthogonal decompositions for \(\mathcal{D}[\tilde{A}/G]\) descend to a semiorthogonal decomposition of \(\Delta\)-equivariant derived categories by showing that they are invariant under the action of \(\Delta\). To that end, notice that there are isomorphisms
  \[
    \tilde{A}^\lambda/C_G(\lambda)\times\Delta\cong A/C_G(\lambda)
  \]
  which give derived equivalences
  \[
    \DD(\tilde{A}^\lambda/C_G(\lambda))^\Delta\cong \DD(A/C_G(\lambda))
  \]
  and
  \[
    \DD[\tilde{A}/G]^\Delta\cong \DD[A/G].
  \]

  Since \(A/G\) is smooth, so is \(P_G/G\). If \(G\) does not act irreducibly on
  \(T_e(P_G)\), then as before the orbifold decompositions will decompose as
  well. So we can assume that \([P_G/G]\) is of types (A), (B), or (C).

  In each of the irreducible cases, the kernels defining Fourier-Mukai functors are supported on:
  \[
    Z_\lambda = \bigcup_{g\in G}g^* \Gamma_\lambda
  \]
  with its reduced structure. By Lemma \ref{lem:delta-equivariant}, the corresponding \(\Delta\)-equivariantized Fourier-Mukai kernels do not intersect. It follows from this that the Fourier-Mukai functors
  \[
  		\Phi_\lambda^\Delta = \Phi_{\mathcal{O}_{Z_\lambda^\Delta}}\colon \mathcal{D}(P_G^\lambda\times A_0/(C(\lambda)\times\Delta))\to \mathcal{D}[\tilde{A}/\Delta\times G]
	\]
	remain fully-faithful, semiorthogonal, and generate the \(\mathcal{D}[A/G]\). Finally, we have the identification
	\[
		P_G^\lambda\times A_0/(C(\lambda)\times\Delta)\cong A^\lambda/C(\lambda)
	\]
	which completes the argument.
\end{proof}

\subsection{Translations}
\label{ssec:trans}

Suppose now that \(G\) does not fix the origin. That is, \(G\) possesses a normal subgroup of translations. Let \(T\subset G\) be the
subgroup generated by translations and \(H\) be the subgroup that fixes the
origin. Notice that \(T\) acts freely and \(G\) admits a semi-direct product
decomposition:
\[
  G = T\rtimes H.
\]
Let \(A_T = A/T\), which is again an Abelian variety. The action of \(H\) on \(A\) descends to an
action on \(A_T\) given by \(h(\overline{a}) = \overline{h(a)}\), where \(a\) is any
preimage of \(\overline{a}\) under the quotient map \(\pi\colon A\to A_T\). This is well-defined as \(T\) is
a normal subgroup. Our goal will be to construct an orbifold semiorthogonal
decomposition of \(\DD[A/G]\) by using the one we know exists for
\(\DD[A_T/H]\). 

In fact, we can prove a more general statement. Suppose a finite group \(G =
K\rtimes H\) acts effectively on a smooth quasi-projective variety \(X\) with \(K\) acting
freely so that \(X_K = X/K\) is also a smooth quasi-projective variety. Then
\(H\) acts on \(X_K\). 

\begin{lemma}
  Pick conjugacy class representatives \( (k_1,h),\ldots,(k_r,h)\) so that \(
  (k,h)\) is conjugate to one of \( (k_i,h)\). Then there is an isomorphism of
  varieties
  \[
    \coprod_{i=1}^r X^{(k_i,h)}/C_G( (k_i,h))\xrightarrow{f} X_K^h/C_H(h).
  \]
  \label{lem:translation-conjugacy-decomp}
\end{lemma}

\begin{proof}
First, we define $f$: for $x\in X^{(k_i,h)}$, we have
\begin{align*}
    (k_i,h) x & = x \\
    (k_i^{-1},1)(k_i,h) x & = (k_i^{-1},1)x \\
    (1,h) x & = (k_i^{-1},1)x. 
\end{align*}
Writing the equality modulo $K$, we have $h(\overline{x}) = \overline{x}$. This defines a map $f'\colon  X^{(k_i,h)} \to X_K^h \to X_K^h/C_H(h)$ for each $k_i$. To factor this map through $X^{(k_i,h)}/C_G( (k_i,h))$ and obtain $f$, it suffices to show that if $(k,h') \in C_G((k_i,h))$ then $h'\in C_H(h)$, but this is straightforward.

It follows from Zariski's main theorem \cite[Th\'eor\`eme 4.4.3]{EGAIII} that a bijective morphism between connected, normal, complex varieties is an isomorphism. Then, it suffices to show that $f$ maps each connected component of the domain bijectively onto its image. 

For surjectivity, we take
  \(\overline{x}\in X_K^h\) and show there exists \( k\in K, h\in H\) and \(x\in X\)
  so that \( (k,h)(x) = x\) and \(\overline{x}\) is conjugate to the image of \(x\)
  under the quotient \(X\to X_K\). Pick any lift \(x\) of \(\overline{x}\). Since
  \(h(\overline{x})=\overline{x}\), \(h(x)\) is in the same orbit as \(x\) under the \(K\)
  action. That is, there exists \(k\in K\) so that \(k(h(x)) = x\). There exists
  \( (k',h')\) so that conjugating \((k,h)\) by \( (k',h')\) is one of the \(
  (k_i,h)\) classes. It follows that \(x\) is conjugate to the image of
  \(k_i(h(x))\) under the action of \(H\). This proves surjectivity.

  Since the action of \(K\) is free, if \(k\neq k'\), then \(X^{(k,h)}\cap X^{(k',h)} = \emptyset\). Indeed, any \(x\) in the intersection would force \(k^{-1}k'h(x) = h(x)\) and hence \(k^{-1}k'\) stabilizes a point. But only the identity stabilizes a point. Hence, we have \(X^{(k_i,h)}\cap X^{(k_j,h)}=\emptyset\)
  for \(i\neq j\). 

	Now if \(x_i\in X^{(k_i,h)}\) and \(x_j\in X^{(k_j,h)}\), then
  the images of \(x_i\) and \(x_j\) are not identified by an element in the
  centralizer of \(C_H(h)\). Indeed, suppose they were, i.e. that there exists
  \(h'\in H\) with \(h'h = h'h\) and \(h'(\overline{x}_i) = \overline{x}_j\). Then
  \(h'(x_i)\in X^{(k_j,h)}\). So
  \[
    x_i=(1,h')^{-1}(k_j,h)(h'(x_i)) = (k_j^{h'},h)(x_i).
  \]
  Hence, \( (k_j,h)\) is conjugate to \( (k_i,h)\) but this cannot happen if
  \(i\neq j\). It follows that disjoint components have disjoint images. 
  
  Now if \(x,y\in X^{(k_i,h)}\) are such that \(\overline{x} = \overline{y}\). Then, since \(K\) acts freely, there
  exists a unique \(k\in K\) such that \(k(x) =y\). We need to see that \(
  (k,1)\) centralizes \( (k_i,h)\). Conjugation of \(
  (k_i,h)\) by \(  (k,1)\) fixes \(x\):
  \[
    (k_i,h)^{(k,1)}(x) = (k,1)^{-1}(k_i,h)(k,1)(x) = (k,1)^{-1}(h_i,k)(y) =
    (k,1)^{-1}(y) = x.
  \]
  We conclude that the fixed locus of \( (k_i,h)^{(k,1)}\) and \( (k_i,h)\)
  intersect. But this only happens if \( (k_i,h)^{(k,1)} = (k_i,h)\) and hence
  \( (k,1)\) centralizes \( (k_i,h)\).
\end{proof}

The following is immediate.

\begin{corollary}
  Let \( (k_1,h),\ldots,(k_r,h)\) be as in Lemma
  \ref{lem:translation-conjugacy-decomp}. Then
  \[
    \DD(X_K^h/C_H(h)) \cong \bigoplus_{i=1}^r \DD(X^{(k_i,h)}/C_G(k_i,h)).
  \]
\end{corollary}

\begin{lemma}
  There is an equivalence of categories:
  \[
    \mathfrak{coh}(X)^G\cong \mathfrak{coh}(X_K)^H
  \]
  induced by the quotient map \(\pi\colon X\to X_K\). This functor can also be
  described by taking \(T\)-invariants. In particular, this equivalence of Abelian categories induces a derived equivalence:
  \[
  	\mathcal{D}[X/G]\cong \mathcal{D}[X_K/H].
  	\]
  \label{lem:semi-direct-free}
\end{lemma}

\begin{proof}
  This is straightforward. If \(\FF\) is a sheaf on \(X\), then a
  \(G\)-equivariant structure is the data of a \(K\)-equivariant structure with
  a compatible \(H\)-equivariant structure. The data of a \(K\)-equivariant
  structure is the data of a pull-back of a sheaf \(\overline{\FF}\) on \(X_K\). The data of a
  compatible \(H\)-equivariant structure is the data of an \(H\)-equivariant
  structure on \(\overline{\FF}\). The details are left to the reader.
\end{proof}

\begin{theorem}
  Suppose \(G = K\rtimes H\) acts on a smooth quasi-projective variety \(X\) so that
  \(K\) acts freely. Consider the induced action of \(H\) on the quotient \(X_K
  = X/H\) and suppose \(\mathcal{D}[X_K/H]\) admits an orbifold semiorthogonal
  decomposition of the form
  \[
    \mathcal{D}[X_K/H] = \langle
    \DD(\overline{X}_K^{h_1}),\ldots,\DD(\overline{X}_K^{h_t})\rangle.
  \]
  Then \(\mathcal{D}[X/G]\) admits an orbifold semiorthogonal decomposition.
  \label{thm:msod}
\end{theorem}

\begin{proof}
  The equivalence of Lemma \ref{lem:semi-direct-free} gives a derived equivalence
  \[
    \mathcal{D}[X/G] \cong \mathcal{D}[X_K/H].
  \]
  The orbifold semiorthogonal decomposition of \(\DD[X_K/H]\) directly induces
  one of \(\DD[X/G]\). The pieces \(\DD(X^h/C_H(h))\) of the orbifold semiorthogonal decomposition of
  \(\DD[X_K/H]\) decompose into \(\DD(X^{(k_i,h)}/C_G(k_i,h))\). Since these are
  pairwise completely orthogonal, any total order suffices. Thus we get a
  semiorthogonal decomposition
  \[
    \DD[X/G] = \langle \bigoplus_{i=1}^{r_1}\DD(X^{(k_i,h_1)}/C_G((k_i,h_1))),
    \ldots, \bigoplus_{i=1}^{r_t}\DD(X^{(k_i,h_t)}/C_G( (k_i,h_t))\rangle.
  \]
  Finally \(\DD(X/G)\) linearity follows from \(\DD(X_K/H)\)-linearity using the
  canonical isomorphism \(X/G\cong X_K/H\).
\end{proof}

Combining Theorem \ref{thm:msod} with our earlier work gives orbifold
semiorthogonal decompositions for all Abelian varieties with smooth quotients.

\begin{theorem}
  Let \(G=T\rtimes H\) be a finite group of automorphisms of an Abelian variety \(A\) such that \(A/G\) is smooth. Then there
  exists an orbifold semiorthogonal decomposition for \(\DD[A/G]\).
  \label{thm:msod-abelian-vars-full}
\end{theorem}

\section{Examples}
\label{sec:examples}

\subsection{Curves}\label{ssec:examples_curves} The one-dimensional case, while it is not covered in this paper, is well studied in the literature. Quotients of an elliptic curve $E$ by a finite group of group automorphisms are weighted projective lines in the sense of Geigle and Lenzing \cite{GL87}. The group $G$ is $\mu_k$ for $k=2,3,4$, or $6$ (hence, they all fall under the analog of Type (A) of Theorem \ref{thm:ALAQClassification}). All quotients have $\mathbb{P}^1$ as a coarse space. There are 4 $\mu_2$ stabilizers for the generic case $k=2$ and 3 stabilizers for the other cases, of orders $(3,3,3)$, $(4,4,2)$, and $(6,3,2)$ respectively. 
The categories $\mathfrak{coh}[E/G]$ are derived equivalent to the module categories over a canonical algebra in the sense of Ringel \cite{Rin90}. The exceptional collections constructed in \cite{Mel95} on $\mathcal{D}[E/G]$ give rise to orbifold semiorthogonal decompositions. 

\subsection{Surfaces}
In this section, we explicitly compute the orbifold semiorthogonal decompositions
constructed here and in \cite{pvdb-equivariant} when \(A\) is a surface and the action of \(G\) fixes \(e\) and acts irreducibly on
\(T_e(A)\). Such an action is one of
three types by Theorem \ref{thm:zero-fixed}. We have already written out Type (C) in Section \ref{sec:type-c}. We begin with Type (A):

\begin{example}
Let \(A\cong E^2\) and \(G = \mu_k^2\rtimes S_2\) for \(k=1,2,3,4,6\). The generic case, \(k=1\), has quotient \(\mathrm{Sym}^2(E)\) and the non-trivial element fixes the diagonal. By \cite[Theorem 4.1]{kuznetsov-perry-cyclic} we have
\begin{equation}
    \label{eq:SODSym2E}
  \mathcal{D}[A/G] = \langle \mathcal{D}(\mathrm{Sym}^2(E)), \mathcal{D}(E)\rangle.
\end{equation}
  
The Abel map $\mathrm{Sym}^2(E) \to J(E)\simeq E$ is a $\mathbb{P}^1$-bundle, hence the decomposition \eqref{eq:SODSym2E} can be refined into one with three indecomposable components equivalent to $\mathcal{D}(E)$ \cite{O92_blowup}.
  
In all other cases, $[A/G]$ admits a full exceptional collection. The coarse quotient is \(\mathrm{Sym}^2(\mathbb{P}^1)\simeq \mathbb{P}^2\). 
  
If \(k=2\), there are 5 
conjugacy classes, whose fixed loci and centralizers are as follows:
  
\begin{center}

\begin{tabular}{l|c|c|c}
Representative $g$ &  $A^g$ & $C(g)$ & $\overline{A^g}$\\
\hline
$(1,1,1)$  & $A$ & $G$ & $\mathbb{P}^2$\\
$(-1,-1,1)$& $E[2]\times E[2]$ & $G$ & $E[2]$\\
$(-1,1,1)$ & $E[2] \times E$ & $ \pair{(-1,1,1),(-1,-1,1)} $ & $E[2]\times \mathbb{P}^1$\\
$(1,1,\sigma)$ & $E$  &$\pair{(1,1,\sigma), (-1,-1,1)}$ & $\mathbb{P}^1$\\
$(1,-1,\sigma)$&$E[2]$&$\pair{(1,-1,\sigma)}$ & $E[2]$\\\hline
\end{tabular}
\end{center}  
  
  Thus we have
  \[
    \mathcal{D}[A/G] = \langle \mathcal{D}(\mathbb{P}^2), 
    \underbrace{\mathcal{D}(\mathbb{P}^1),...,\mathcal{D}(\mathbb{P}^1)}_{5 \mbox{ copies}},
    \underbrace{\mathcal{D}(\mathrm{pt}),...,\mathcal{D}(\mathrm{pt})}_{8 \mbox{ copies}}\rangle.
  \]
  
With a similar computation we find the orbifold semiorthogonal decomposition in the other cases. For $k=3$ there are 9 conjugacy classes. The corresponding orbifold semiorthogonal decomposition has one two-dimensional component, 3 one-dimensional components, and 5 zero-dimensional components:
 \[
    \mathcal{D}[A/G] = \langle \mathcal{D}(\mathbb{P}^2),
    \underbrace{\mathcal{D}(\mathbb{P}^1),...,\mathcal{D}(\mathbb{P}^1)}_{7 \mbox{ copies}},
    \underbrace{\mathcal{D}(\mathrm{pt}),...,\mathcal{D}(\mathrm{pt})}_{21 \mbox{ copies}}\rangle.
  \]

For $k=4$ there are 14 conjugacy classes. The corresponding orbifold semiorthogonal decomposition has one two-dimensional component, 4 one-dimensional components, and 9 zero-dimensional components:
\[
    \mathcal{D}[A/G] = \langle \mathcal{D}(\mathbb{P}^2), 
    \underbrace{\mathcal{D}(\mathbb{P}^1),...,\mathcal{D}(\mathbb{P}^1)}_{6 \mbox{ copies}},
    \underbrace{\mathcal{D}(\mathrm{pt}),...,\mathcal{D}(\mathrm{pt})}_{33 \mbox{ copies}}\rangle.
  \]


For $k=6$ there are 27 conjugacy classes.  The corresponding orbifold semiorthogonal decomposition has one two-dimensional component, 6 one-dimensional components, and 20 zero-dimensional components:
\[
    \mathcal{D}[A/G] = \langle \mathcal{D}(\mathbb{P}^2), 
    \underbrace{\mathcal{D}(\mathbb{P}^1),...,\mathcal{D}(\mathbb{P}^1)}_{ \mbox{7 copies}},
    \underbrace{\mathcal{D}(\mathrm{pt}),...,\mathcal{D}(\mathrm{pt})}_{ \mbox{26 copies}}\rangle.
  \]

\end{example}

For type (B):

\begin{example}Let $E$ be an elliptic curve, and let $A$ be the abelian surface defined by $x_1+x_2+x_3=0$ in $E^3$. Identify $A$ with $E^2$ by eliminating the variable $x_3$. Then, $S_3$ acts on $E^2$ by permuting variables, and we consider \([E^2/S_3]\). The quotient is \(\mathbb{P}^2\) by the classification \cite[Theorem 1.1]{ga-smooth2}. The group \(S_3\) acts
  on \(T_e(E^2)\) by 
  \[
    (1,2) = \begin{pmatrix}
      0 & 1 \\ 
      1 & 0 \\
    \end{pmatrix},\ 
    (1,2,3) = \begin{pmatrix}
      0 &-1 \\
      1 &-1 
    \end{pmatrix}.
  \]
The conjugacy class corresponding to \( (1,2)\) fixes the diagonal
  \(\Delta\subset E\times E \cong A\). The conjugacy class corresponding to \(
  (1,2,3)\) fixes the anti-diagonal copy of the three-torsion subgroup \(E[3]\). The centralizer of $(1,2)$ is $\left\langle (1,2)\right\rangle$, and similarly $C((1,2,3))= \left\langle (1,2,3)\right\rangle$. Then, the centralizers act trivially on the fixed loci. By
  Theorem \ref{thm:fixing-origin}, there is an orbifold semiorthogonal decomposition of
  the form
  \[
    \DD[E^2/S_3] = \langle \DD(\mathbb{P}^2), \DD(E), \DD(E[3])\rangle.
  \]
\end{example}


\providecommand{\bysame}{\leavevmode\hbox to3em{\hrulefill}\thinspace}
\providecommand{\MR}{\relax\ifhmode\unskip\space\fi MR }
\providecommand{\MRhref}[2]{%
  \href{http://www.ams.org/mathscinet-getitem?mr=#1}{#2}
}
\providecommand{\href}[2]{#2}

\end{document}